\documentclass[12pt]{amsart}
\usepackage{fullpage,url,amssymb,enumerate,colonequals}
\usepackage{mathrsfs}
\usepackage[section]{placeins}
\usepackage{MnSymbol}
\usepackage{listings}
\usepackage{extarrows}
\usepackage{lscape}
\usepackage[all,cmtip]{xy}
\usepackage{graphicx} 
\usepackage{todonotes}
\usepackage[indLines=true, noEnd=false]{algpseudocodex}
\usepackage[OT2,T1]{fontenc}
\usepackage{color}
\usepackage[
        colorlinks, citecolor=darkgreen,
        backref,
        pdfauthor={Nuno Freitas},
]{hyperref}
\usepackage{comment}
\usepackage{multirow}

\numberwithin{equation}{section}

\newtheorem{lemma}[equation]{Lemma}
\newtheorem{theorem}[equation]{Theorem}
\newtheorem{proposition}[equation]{Proposition}

\newtheorem{corollary}[equation]{Corollary}

\newtheorem{claim*}{Claim}

\newtheorem{defn}[equation]{Definition}
\newtheorem{alg}[equation]{Algorithm}

\theoremstyle{definition}
\newtheorem{remark}[equation]{Remark}

\newtheorem{example}[equation]{Example}

\definecolor{darkgreen}{rgb}{0,0.5,0}
\definecolor{rem}{rgb}{0.8,0,0}
\definecolor{new}{rgb}{0.3,0.1,0.9}
\definecolor{reply}{rgb}{0,0,0.8}
\definecolor{gray}{gray}{0.7}

\renewcommand{\det}{\operatorname{det}}
\renewcommand{\gcd}{\text{gcd}}
\newcommand{\F}{\mathbb{F}}
\newcommand{\PP}{\mathbb{P}}
\newcommand{\Q}{\mathbb{Q}}

\newcommand{\rhobar}{{\overline{\rho}}}

\newcommand{\calO}{\mathcal{O}}
\newcommand{\fp}{\mathfrak{p}}

\newcommand{\Gal}{\text{Gal}}
\newcommand{\Frob}{\text{Frob}}

\newcommand{\GL}{\operatorname{GL}}
\newcommand{\SL}{\operatorname{SL}}
\newcommand{\SD}{\operatorname{SD}}

\newcommand{\Fbar}{{\overline{F}}}

\renewcommand{\ker}{\operatorname{ker}}

\DeclareMathOperator{\Ind}{Ind}

\DeclareMathOperator{\Art}{\operatorname{Art}}

\DeclareMathOperator{\Image}{Im}
\DeclareMathOperator{\Center}{Center}

\providecommand{\cl}[1]{\mathcal{#1}}

\providecommand{\p}{\mathfrak{p}}
\providecommand{\ger}[1]{\mathfrak{#1}}

\providecommand{\barra}[1]{\overline{#1}}
\providecommand{\tenso}{\otimes}

\providecommand{\Q}{\mathbb Q}

\providecommand{\ma}{\mathbb}

\title{On inertial types of elliptic curves}

\date{\today}

\author{Jose Castro-Moreno}
\address{Instituto de Ciencias Matem\'aticas (ICMAT),
          Nicol\'as Cabrera 13-15
         28049 Madrid, Spain}
\email{joseantonio.castro@icmat.es}

\author{Enric Florit}
\address{Universitat Oberta de Catalunya
Rambla del Poblenou, 156, 08018, Barcelona, Spain
         }
\email{efloritz@uoc.edu}

\author{Nuno Freitas}
\address{Instituto de Ciencias Matem\'aticas (ICMAT),
          Nicol\'as Cabrera 13-15
         28049 Madrid, Spain}
\email{nuno.freitas@icmat.es}

\thanks{Castro-Moreno PhD thesis project received the support of a fellowship from the ``La Caixa'' Foundation (ID 100010434) The fellowship code is LCF/BQ/DR24/12080026.}
\thanks{Florit was supported by the Spanish Ministry of Universities (FPU20/05059) and by grants PID2022-137605NB-I00 and 2021 SGR 0146.}
\thanks{Freitas was partly supported by the PID 2022-136944NB-I00 grant of the MICINN (Spain)}

\subjclass[2010]{}

\begin{document}

\begin{abstract}
    We classify the inertial Weil-Deligne types arising from elliptic curves over all finite extensions $F/\Q_p$. Based on this classification, we give a fully explicit description of the types and implement an algorithm that computes all inertial types of elliptic curves defined over a given $F$. As an application, we determine all inertial types arising from elliptic curves over any extension $F/\Q_p$ of degree at most 3.
\end{abstract}

\maketitle

\section{Introduction}\label{S: Intro}

Let $p$ be prime and $F / \Q_p$ be a finite extension with algebraic closure
$\Fbar$. We denote the Weil group of~$F$ by $W_F \subset \Gal(\Fbar / F)$ and
we write $I_F \subset W_F$ for its inertia subgroup. For the purpose of this
introduction, given a representation $\rho : W_F \to \GL_2(\ma C)$, we will
call the restriction $\tau : = \rho|_{I_F}$ {\it an inertial type} (see
Section~\ref{S: Background} for precise definitions).

In the work~\cite{dembélé2024galoisinertialtypeselliptic} of the third author
with Dembélé and Voight, the authors give a completely explicit description of
the inertial types arising from the $\ell$-adic representations~$\rho_{E,\ell}$
attached to elliptic curves over $F = \Q_p$ for $p \neq \ell$. As discussed in
    {\it loc. cit.}, the study of inertial types has many applications, including
to modularity, to the classification of representations of $G_F$, and to
Diophantine equations. Moreover, a recent application of the results
in~\cite{dembélé2024galoisinertialtypeselliptic} appeared in the work of
Bosco~\cite{bosco20233adicrepresentationsarisingelliptic} where the
$(\varphi,\Gal(K/\Q_3))$-modules associated to $3$-adic representations of
elliptic curves over~$\Q_3$ are explicitly computed.

The goal of this paper is to extend the results
of~\cite{dembélé2024galoisinertialtypeselliptic} over~$\Q_p$ to any~$F$ as
above. As explained in Section~\ref{S: Uniform cases}, the tame inertial types
and those arising from elliptic curves with potentially multiplicative
reduction are well understood, so our focus are inertial types arising from
elliptic curves with potentially good reduction over 2-adic and 3-adic fields.
When working with this generality, we are faced with new important phenomena.
Specifically, for extensions~$F/\Q_2$, exceptional types can have image the
quaternion group~$Q_8$ or $\SL_2(\F_3)$ (instead of only the latter) and {\it
        triply imprimitive} types exist, i.e., types that can be obtained by induction
of characters from three different quadratic extensions of $F$. Furthermore,
the number of inertial types increases with the degree of~$F$ whilst over
$F/\Q_p$ with $p \geq 5$ the types are limited and can be described uniformly.

We will now state our main results. Let $C_n$ denote the cyclic group with~$n$
elements.

\begin{theorem}\label{T: Main Theorem for p=3}
    Let $F/\ma Q_3$ be  a finite extension and $\tau : I_F \to \GL_2(\ma C)$ an inertial type defined over~$F$. Then,~$\tau$ is the inertial type of an elliptic curve with potentially good reduction if and only if it has trivial determinant and one of the following holds.
    \begin{enumerate}[(i)]
        \item $\tau$ is a principal series type with image isomorphic to~$C_2$, $C_3$, $C_4$ or $C_6$;
        \item $\tau$ is a  unramified supercuspidal type with image isomorphic to~$C_3$, $C_4$ or $C_6$;
        \item $\tau$ is a ramified supercuspidal type with image isomorphic to~$C_3\rtimes C_4$.
    \end{enumerate}
\end{theorem}

To state and prove our 2-adic theorem, it is helpful to introduce some
additional notation. Let~$F/\Q_2$ and~$L/F$ be finite extensions. We will say
that $L/F$ satisfies property {\bf (H1)} when
\begin{itemize}
    \item $L/F$ is a cubic Galois extension if $\mu_3\subset F$, or
    \item $L/F$ is the Galois closure of a ramified cubic extension $L'/F$ if $\mu_3\not\subset F$.
\end{itemize}

Let $L/F$ satisfy {\bf (H1)} and $K_i = L(\sqrt{d_i})$ be different quadratic
extensions of $L$ for $i=1,2,3$. We will say that the fields $K_i$ satisfy
property {\bf (H2)} if the $d_i$ generate a $2$-dimensional irreducible Galois
module of $L^{\times}/(L^{\times})^2$ under the action of $\Gal(L/F)$.
\begin{remark}
    If $\mu_3\subset F$, then any cubic extension is normal and there are exactly $4$ cubic extensions satisfying {\bf (H1)}. However, when $\mu_3\not\subset F$ the Galois closure of the three cubic ramified extensions $L'/F$ is the same, so property {\bf (H1)} uniquely determines $L/F$ and $\Gal(L/F) \simeq S_3$.
\end{remark}

Our 2-adic theorem is the following.

\begin{theorem}\label{T: Main Theorem for p=2}
    Let $F/\ma Q_2$ be  a finite extension and $\tau : I_F \to \GL_2(\ma C)$ an inertial type defined over~$F$. Then,~$\tau$ is the inertial type of an elliptic curve with potentially good reduction if and only if $\tau$ has trivial determinant and one of the following holds.
    \begin{enumerate}[(i)]
        \item $\tau$ is a principal series type with image isomorphic to $C_2$, $C_3$, $C_4$ or $C_6$;
        \item $\tau$ is a unramified supercuspidal type with image isomorphic to $C_3$, $C_4$ or $C_6$;
        \item $\tau$ is a ramified supercuspidal type with image isomorphic to~$Q_8$ and such that:\begin{enumerate}[(a)]
                  \item if $\mu_3\subset F$, then $\tau$ is triply imprimitive,
                  \item if $\mu_3\not\subset F$, then $\tau$ is simply imprimitive;
              \end{enumerate}
        \item There is an extension~$L/F$ satisfying~{\bf (H1)} and a
              representation~$\rho:W_L\xrightarrow{}\GL_2(\ma C)$ such that:\begin{enumerate}
                  \item $\rho\vert_{I_L}=\tau\vert_{I_L}$,
                  \item $\Image\rho=\Image\tau_L\simeq Q_8$,
                  \item $\rho$ is triply imprimitive induced from $K_1$, $K_2$ and~$K_3$ satisfying {\bf (H2)}, and
                  \item  $\rho^{\sigma}\simeq\rho$ for all $\sigma\in \Gal(L/F)$, where $\rho^{\sigma}(t)=\rho(\sigma^{-1}t\sigma)$.
              \end{enumerate}
              In this case $\tau$ is exceptional.
    \end{enumerate}
\end{theorem}
We observe that there is no condition on the conductor of $\tau$ in the two theorems above whilst there is the well known bound $v_F(N_E) \leq 2+3v_F(3)+5v_F(2)$ for the conductor $N_E$ of elliptic curves $E/F$  (see~\cite[Ch.~IV, Thm~10.4]{Silverman1994}). This has the surprising consequence that any type satisfying our theorems automatically satisfies this bound as well (see also Remark~\ref{R: Q8conductor}).

Let us also highlight Theorem~\ref{Inertia field determine the type}, as it is crucial to obtain the above results; it states that an inertial type $\tau$ arising on an elliptic curves is determined by its kernel. In other words, the extension of $F^{un}$ cut out by $\tau$ determines $\tau$.

 We have also made the above results computationally accessible. To this end, we
first give equivalent formulations of our main theorems in terms of inertia
characters, namely, Theorems~\ref{T: explicit conditions p=3} and~\ref{T:
    explicit conditions p=2}, respectively. Then, based on this formulation, we
discuss in Section~\ref{S: algo} an algorithm that determines a complete
classification of inertial types arising from elliptic curves defined over a
given~$F$; this algorithm has been fully implemented in {\tt
        Magma}~\cite{magma}. As an application, for $p=2$ and $p=3$, we computed and
tabulated all inertial types arising from elliptic curves defined over all
quadratic and cubic extensions of~$\Q_p$; the algorithm can be found in~\cite{MAGMAFiles} and the list of types in~\cite{InTypeWebPage}. Since the description of types for $p \geq
    5$ is uniform and well known, this effectively creates a database of all
inertial types arising from elliptic curves over any $F$ with $[F:\Q_p] \leq
    3$; to keep expanding this database it suffices to run our algorithm once over
any new field~$F$.

The Galois representations arising on the $\ell$-adic Tate module of an
elliptic curve $E$ defined over extensions $F/\Q_p$ with $p \neq \ell$ have
been studied by several authors, including the work of
Dokchitser--Dokchitser~\cite{Dokchitser_2008} and
Coppola~\cite{NirvannaWild3-adic, NirvannaWild2-adic, NirvannaWildCyclic}.
These works provide a description of the representation of inertia as an
abstract group together with the image of Frobenius; moreover, in most cases,
they describe a recipe to compute the field fixed by their kernel, from a given
Weierstrass model for $E$. A key difference between these works and the
approach we take here is that we analyze inertial types using the
classification of local Galois representations into principal series, Steinberg
and supercuspidal representations (as done
in~\cite{dembélé2024galoisinertialtypeselliptic}). This allows us to split the
types not in terms of its image as an abstract group but rather in terms of how
they arise; for instance, there are inertial types with image $Q_8$ that are
both ramified supercuspidal and exceptional and they need to be treated
differently in order to obtain a complete explicit description. Moreover, for a
concrete field~$F$ our algorithm classifies all types over~$F$, hence we know
exactly how many inertial types there are at each conductor and what is the
fixed fields by their kernels; this level of detail has applications to
Diophantine equations as discussed at length in~\cite{BCDFmulti}. It is also
common in Diophantine applications to have partial information about the
inertial type, for example, that it is supercuspidal induced from a certain
quadratic $K/F$. Our treatment allows to narrow the search by including
additional constraints lowering drastically the amount of computations needed,
which makes possible to find all inertial types with these constrains even for
large fields.

We want to highlight the following consequence of our approach. An inertial
type arising from an elliptic curve can also appear, for example, in a
classical or Hilbert modular form whose field of coefficients is not $\Q$ (and
so not corresponding to an elliptic curve by the Eichler--Shimura
correspondence). Since our description is independent of elliptic curves, it
can be used to study inertial types coming from other sources, because any type
satisfying the conditions in our main theorems is computed by our algorithm. An
application of this can be found in the work of
Pacetti--Torcomian~\cite{PacTor}, where they use our classification to
determine inertial types of Jacobians of hyperelliptic curves appearing in the
study of the Fermat equation $x^5 + y^p = z^3$; see Example~\ref{Example Ariel}
for more details.

To conclude this introduction, let us briefly describe the main differences
between our approach and that of~\cite{dembélé2024galoisinertialtypeselliptic}.
Firstly, the results in {\it loc. cit.} used to study nonexceptional types are
very explicit as the authors could rely on the specific fields they were
working with to describe concrete generators of their groups of units. Instead,
we generalize and present the necessary results in a unified way over any~$F$,
concluding that for $p=3$ all fields behave similarly and for that $p=2$ there
is a clear splitting between the fields containing the third roots of unity and
those that do not. Secondly, the study of the exceptional types in {\it loc.
        cit.} relies on the classification of $\GL_2(\F_3)$-extension of~$\Q_2$ due to
Bayer and Rio~\cite{BayRio}. In the absence of analogous results over~$F$, we
first base change an exceptional type to a suitable cubic extension of~$F$
where it becomes nonexceptional and we can apply the previously developed
theory, and then we descend it back to $F$. Lastly and more importantly, we
prove that all the inertial types we describe and the algorithm computes do
indeed arise from elliptic curves, without the need to exhibit a curve for each
type as done in~\cite{dembélé2024galoisinertialtypeselliptic}. This is an
interesting result on its own but also avoids a computationally heavy step in
the algorithm; indeed, since the number of types increases rapidly with the
degree of $F$, finding a curve for each type becomes impractical very quickly
(see also Remarks~\ref{R: comments on algorithm} and \ref{R: Curves are bad}).

\subsection{Outline} In Sections~\ref{S: Background}--\ref{S: section 3}, we recall basic results
and present some of our results about inertial types and elliptic curves. In
Section~\ref{S: Uniform cases}, we classify inertial types of elliptic curves
over any $F/\Q_p$ for $p \geq 5$. In Section~\ref{S: Necessary conditions}, we
present necessary conditions that inertial types of elliptic curves must
satisfy and, in Section~\ref{S: sufficient conditions}, we show that this
conditions are indeed sufficient. In Section~\ref{S: charactersformulation}, we
present how these conditions can be expressed purely in terms of characters.
Finally, Section~\ref{S: algo} discusses the algorithm that we have implemented
in {\tt Magma} together with some examples and remarks.

\subsection{Computational software}
We have developed a {\tt Magma} package to compute and work with inertial
types, see Section~\ref{S: algo} for details. All our code is available at
\cite{MAGMAFiles}, and the repository contains instructions and examples of how
it is used.

\subsection{Acknowledgements}
The authors would like to thank John Jones for pointing out the algorithms
in~\cite{Guardia} that find nice polynomials that help to describe some
inertial fields. They were extremely useful in the early development of our
    {\tt Magma} package. Enric Florit thanks the ICMAT and the number theory seminar at UAM for their hospitality in October 2025.

\section{Background material}\label{S: Background}

In this section, we recall the basic definitions and results needed for the
classification of inertial types of elliptic curves; we refer to ~\cite[\S
    2]{dembélé2024galoisinertialtypeselliptic} for a more detailed exposition.

\subsection{Notation}\label{Notation}

Let $p$ be a prime. Let $F$ be a finite extension of $\ma Q_p$, $\cl O_F$ its
ring of integers, $\p$ its maximal ideal, $\pi\in \p$ a uniformizer and
$v_F:F^{\times}\xrightarrow{}\ma C$ the valuation of $F$ normalized so that
$v_F(\pi)=1$. We denote the residue field $\calO_F/\fp$ by~$\ma F_q$ and by
$\overline{F}$ a fixed algebraic closure of~$F$. For any finite extension
$K/F$, let $K^{un} \subset \Fbar$ be the maximal unramified extension of~$K$;
we have $K^{un}=K F^{un}$.

Let $W_F$ be the Weil group of $F$ and $I_F$ its
inertia subgroup. The group $W_F$ is made into a topological group by declaring
that $I_F$ is open in $W_F$ and that the subspace topology of $I_F$ in~$W_F$ is
the same as the subspace topology of $I_F$ in $\Gal(\overline{F}/F)$. We denote by $W_{F}^{ab}$ be the maximal abelian quotient of $W_F$ and
$\Art_F:F^{\times}\xrightarrow{\sim} W^{ab}_F$ the local Artin reciprocity map
normalized so that $\Art(\pi)=\operatorname{\Frob_F}$ is the (class of a)
geometric Frobenius, i.e.,
$\operatorname{\Frob_F}(x^q)\equiv x \pmod{\p}$.

For a topological group $G$, by a character of~$G$ we mean a continuous group homomorphism $\theta:G\xrightarrow{} \ma C^{\times}$ and, given a morphism of
topological groups $\varphi:G\xrightarrow{}G'$, we say that $\theta$ factors
through $\varphi$ if there exists a character
$\theta':G'\xrightarrow[]{}\ma C$ such that
$\theta=\theta'\circ\varphi$; when $\varphi$ is understood, for example, when
$G'$ is a quotient of $G$ we will simply say that $\theta$ factors through
$G'$.

We will denote by $\operatorname{Ord}(\theta)$ the order
of a character. A character $\theta$ of $W_{F}$ must factor through
$W_{F}^{ab}$ and so we can use the Artin map to identify it with a character of
$F^{\times}$, denoted by $\theta^{A}:=\theta\circ \Art$. Given a
character~$\theta$ of~$W_F$, its conductor is the biggest ideal~$\p^m$ so that
$\theta^{A}\vert_{1+\p^m}$ is trivial, and so~$\theta^{A}\vert_{\cl
        O^{\times}_F}$ factors through
$\cl O_F^{\times}/(1+\p^m)\simeq(\cl O_F/\p^m)^{\times}$; we will denote the
exponent~$m$ as $m(\theta)$ and the conductor as
$\operatorname{Cond}(\theta)$, so that
$\operatorname{Cond}(\theta)=\p^{m(\theta)}$ and
$m(\theta)=v(\operatorname{Cond}(\theta))$.

Let $\omega: W_F \to \ma C^\times$ be the character corresponding via the Artin map to the norm
character~$|\cdot|_v$ on $F$, so that~$\omega(g)=q^{-a}$ for
$g|_{F^{un}}=\Frob^{a}$ with $a \in \ma Z$. Given a quadratic extension
$K/F$ we denote by $\psi_K$ the unique quadratic character of $W_F$ whose
kernel is $W_K$ and by $\varepsilon_K$ its restriction to inertia
$\varepsilon_K=\psi_K\vert_{I_F}$.

Finally we will denote by $Q_8$ the quaternion
groups of 8 elements and by $\SD_{16}$ the 2-sylow subgroup of
$\GL_2(\F_3)$.

\subsection{Weil--Deligne representations} \label{WDrepresentation}
\begin{defn}
    A $2$-dimensional Weil--Deligne (WD) representation is a pair $(\rho,N)$ where
    \begin{enumerate}[(i)]
        \item $\rho:W_F\longrightarrow\GL_2(\ma C)$ is a continuous representation,
        \item $N\in \GL_2(\mathbb{C})$ is a nilpotent operator,
        \item $\rho(g)N\rho(g)^{-1}=\omega(g)N$.
    \end{enumerate}
\end{defn}

We say that two Weil--Deligne representations $(\rho,N),(\rho',N')$ are
isomorphic if there exist $P\in \GL_2(\ma C)$ such that
$P\rho(g)P^{-1}=\rho(g)'$ for all $g$ and $PNP^{-1}=N'$.

A key concept in the study of Weil--Deligne representation is its conductor.
The definition of conductor for characters (i.e dimension 1 representations) is
the one given in~\S\ref{Notation}, and for general representation we refer
to~\cite[Chapter IV,\S 10]{Silverman1994}. For a WD-representation~$\rho$, we
denote its conductor and conductor exponent by $\operatorname{Cond}(\rho)$ and
$m(\rho)$, respectively.

Any 2-dimensional Weil--Deligne representation $\rho$ of~$W_F$ is isomorphic to
one of the following:
\begin{enumerate}[(i)]
    \item {\it Principal Series}: Take $N=0$ and
          $\rho=PS(\theta_1,\theta_2):=\theta_1\oplus\theta_2$,
          where $\theta_1,\theta_2 : W_F\xrightarrow{}\ma C$ are characters such that $\theta_1\theta_2^{-1}\neq
              \omega^{\pm1}$. We have $m(\rho)=m(\theta_1)+m(\theta_2)$.
    \item {\it Special or Steinberg}: Take $N=\left(\begin{smallmatrix} 0&1\\0&0\end{smallmatrix}\right)$ and
          $\rho=\operatorname{St}(\theta)$, where
          $\operatorname{St}(\theta)=\theta\oplus\omega\theta$ for a character $\theta:W_F\xrightarrow{} \ma C$. Here $m(\rho)=2m(\theta)$ if $\theta$ is ramified and $m(\rho)=1$ otherwise.
    \item {\it Supercuspidal representations}: Take $N=0$
          and $\rho$ an irreducible $2$-dimensional representation. If the projective
          image of $\rho$ is dihedral, then $\rho$ is called {\it imprimitive} or              {\it nonexceptional}, otherwise it is called {\it primitive} or {\it exceptional},
          in which case its projective image is $A_4$ or $S_4$. The exceptional types
          only occur when $p=2$. The conductor formula for the imprimitive case is given in~\eqref{E:condBC} below.
\end{enumerate}

The following result will come handy when dealing with twists of
representations.

\begin{lemma}\label{conductor of twists}
    Let $(\rho,N)$ be a 2-dimensional Weil--Deligne representation of $W_F$, such that $N=0$ and let $\theta$ be a character of $W_F$. If $m(\rho)\neq 2m(\theta)$, then $m(\theta\tenso \rho)=\operatorname{max}\{m(\rho),2m(\theta)\}$
\end{lemma}
\begin{proof}
    For the supercuspidal case see~\cite[Proposition 3.4]{Tunnell1978}, the principal series follows from a simple computation.
\end{proof}

\subsection{Nonexceptional supercuspidal representations}\label{S: Supercuspidal}

Let $K/F$ be a quadratic extension, $\psi_K : W_F\xrightarrow{}\ma C$ the
quadratic character associated to~$K$, that is, $\ker \psi_K =W_K$ and
$\varepsilon_K:=\psi_K\vert_{I_F}$. Let $s\in W_F$ be a lift of the non-trivial
element of $\Gal(K/F)$ and, for a character $\theta:W_K\xrightarrow{}\ma C$, we
define $\theta^{s}(g)=\theta(s^{-1}gs)$; since $W_K$ is normal in $W_F$ this is
independent of the choice of~$s$. By local class field theory
$(\theta^s)^A=\theta^A\circ s$.

Let $\theta:W_K\xrightarrow{} \ma C$ be a character satisfying~$\theta\neq
    \theta^{s}$. The representation $\rho=\Ind_{W_K}^{W_F} \theta$ is irreducible
and we call $(\rho,0)$ the {\it nonexceptional supercuspidal representation
        induced from/by~$\theta$}. Its conductor exponent satisfies
\begin{equation} \label{E:condBC}
    m(\rho) = \begin{cases}
        2m(\chi),            & \text{ if $K/F$ is unramified;} \\
        m(\chi) + m(\psi_K), & \text{ if $K/F$ is ramified. }\
    \end{cases}
\end{equation}

For convenience we recall the following auxiliary lemmas.

\begin{lemma}\label{Factor through the norm}
    The following are equivalent:
    \begin{enumerate}[(i)]
        \item $\theta^s=\theta$,
        \item $s(x)/x\in \ker \theta^{A}$ for all $x\in K^{\times}$,
        \item $\theta^{A}$ factors through the norm map $\operatorname{Nm_{K|F}}:K^{\times}\xrightarrow{}F^{\times}$.
    \end{enumerate}
\end{lemma}
\begin{proof}
    See~\cite[Lemma 2.3.2]{dembélé2024galoisinertialtypeselliptic}
\end{proof}

\begin{lemma}\label{cyclotomicdeterminant}
    Let $K/F$ be a quadratic extension and $\psi_K$ the associated quadratic character. Let $\delta$ be the determinant of the 2 dimensional representation of $W_F$ induced by a character $\theta:W_K\xrightarrow{}\ma C$. Then $\delta$ is unramified if and only if $\theta^{A}\vert_{\cl O_F^{\times}}=\psi_K^{A}\vert_{\cl O_F^{\times}}=\varepsilon_K^{A}$. In which case $\theta^{s}\vert_{I_K}=\theta^{-1}\vert_{I_K}$.
\end{lemma}
\begin{proof}
    In~\cite[Lemma 2.3.7]{dembélé2024galoisinertialtypeselliptic} the result appears as an implication but the proof shows that it is and equivalence.
\end{proof}

\begin{corollary}\label{cor: the inductions are really irreducible}
    Let $K/F$ be a quadratic extension, $\theta$ a character of~$W_K$ and $\rho := \Ind_{W_K}^{W_F} \theta$.
    Assume that $\det~\rho$ is unramified. If $\theta|_{I_K}$ has order $\geq 3$, then $\rho$ is irreducible. Moreover if $K/F$ is ramified then $\rho\vert_{I_F}$ is also irreducible.
\end{corollary}
\begin{proof} Let $\chi := \theta|_{I_K}$. From Lemma~\ref{cyclotomicdeterminant} we have $\chi^s = \chi^{-1}$.
    Since $\chi$ has order $> 2$, we have $\chi^{-1} \neq \chi$ and hence $\theta^s \neq \theta$, thus $\rho=\Ind_{W_K}^{W_F} \theta$ is irreducible. If $K/F$ is ramified the same reasoning applies to $\rho\vert_{I_F}=\Ind_{I_K}^{I_F} \chi$.
\end{proof}

Since the determinant of a representation coming from an elliptic curve is
$\omega$ which is unramified, the characters inducing the inertial type must
satisfy $\theta^A\vert_{\cl O_F^{\times}}=\varepsilon_K^{A}$. The following
lemma shows that this condition can be reduce to a finite number of
computations.

\subsection{Exceptional (supercuspidal) representations}
\label{sec:exceptional}

The supercuspidal representations that are not induced from a character are
exceptional and their projective image is $A_4$ or $S_4$. Since exceptional
representations only occur for $p=2$, we assume in this section that $p=2$.

\begin{proposition}\label{triply imprimitive}
    Let $\rho$ be a supercuspidal representation and let $\ger J(\rho)$ be the group of characters $\theta$ such that $\rho\otimes\theta\simeq\rho$. Then $\#\ger J(\rho)\in\{1,2,4\}$ and the following statements hold:
    \begin{enumerate}[(i)]
        \item  $\# \ger J(\rho)=1 \iff \rho$ is exceptional.
        \item  $\# \ger J(\rho)=2\iff$ There is a unique quadratic extension $K/F$ such that $\rho=\Ind_{W_K}^{W_F} \theta $.
        \item  $\# \ger J(\rho)=4\iff$ the projective image of $\rho$ is $D_2\simeq C_2\times C_2$  $\iff\rho=\Ind_{W_K}^{W_F} \theta $ can be induced from $3$ distinct quadratic extensions$\iff$ the character $\theta^{s}/\theta$ factors through the norm map.
    \end{enumerate}

    Moreover if $\# \ger J(\rho)=4$ then $\ger J(\rho)$ is the group formed by the
    trivial character and the $3$ characters $\psi_K$ associated to the $3$
    quadratic extensions $K$ of $F$ from which $\rho$ is induced.
\end{proposition}
\begin{proof}
    See~\cite[41.3 Corollary]{Bushnell2006} and~\cite[2.7 Proposition]{Gerardin}.
\end{proof}

According to this proposition, we call the representations $\rho$ with $\# \ger
    J(\rho)=2$ {\it simply imprimitive} and those with $\# \ger J(\rho)=4$ we call
them {\it triply imprimitive}.

\begin{proposition}\label{P: base change of exceptional types}
    For $L/F$ a finite extension and $\rho$ an exceptional supercuspidal representation, we define
    $\rho_L:=\rho\vert_{W_L}$. Then,

    \begin{itemize}
        \item There exist a cubic extension $L/F$ such that $\rho_L$ is non exceptional.
        \item If $L/F$ is cubic and Galois, then $\rho_L$ is triply imprimitive and the
              nontrivial elements of $\ger J(\rho_L)$ are permuted transitively by
              $\Gal(L/F)$.
        \item If $L/F$ is cubic and non-Galois, let $\tilde{L}/F$ be the normal closure of
              $L/F$ and $E/F$ the maximal unramified sub-extension of $\tilde{L}/F$. Then,
              $\rho_L$ is simply imprimitive, $\rho_{\tilde{L}}$ is triply imprimitive and
              $\rho_E$ is primitive; moreover, the nontrivial elements of $\ger
                  J(\rho_{\tilde{L}})$ are permuted transitively by $\Gal(\tilde{L}/F)$.
    \end{itemize}
\end{proposition}
\begin{proof}
    See \cite[\S 42.2, Theorem and Remark]{Bushnell2006}.
\end{proof}

\subsection{Inertial types}
\label{sec:typesCurves}

A {\it Weil--Deligne inertial type} is an equivalence class $[\rho,N]$ of
WD-representations under the equivalence relation $(\rho,N)\sim(\rho',N')$ if
there exist $P\in \GL_2(\ma C)$ such that $P\rho(g)P^{-1}=\rho'(g)$ for all
$g\in I_F$ and $PNP^{-1}=N'$. Thus one can think of inertial types simply as
WD-representations of $I_F$. Since $N=0$ except for Steinberg representations,
that are classified in Lemma~\ref{Pot Mult Uniform}, we will denote the
equivalence class $[\rho,N]$ simply by~$\tau :=\rho\vert_{I_K}$.

\begin{defn}
    Let $\tau:=\rho\vert_{I_F}$ be an inertial type. We call $\tau$ principal series (respectively Steinberg) inertial type if $\rho$ a principal series (respectively Steinberg) Weil-Deligne representation. We call $\tau$ a (un)ramified supercuspidal inertial type if $\rho$ is a Weil-Deligne representation that is induced from a (un)ramified quadratic extension $K/F$. Finally we call $\tau$ an exceptional inertial type if $\rho$ is an exceptional Weil-Deligne representation.
\end{defn}

Considering the restrictions to inertia of the different types of Weil--Deligne
representations and imposing that $\det \tau =1$, we obtain that any
2-dimensional inertial type with trivial determinant is of one of the
following:

\begin{enumerate}[(i)]
    \item If $\tau_E=PS(\theta_1,\theta_2)\vert_{I_F}$ then
          $\theta_1\vert_{I_F}=\theta_{2}^{-1}\vert_{I_F}$ and, letting
          $\chi:=\theta_1\vert_{I_K}$, we have
          \begin{equation}\label{m for PS}
              \tau_E =\chi\oplus\chi^{-1} \quad \text{and} \quad m(\tau_E)=2m(\chi).
          \end{equation}
    \item If $\tau_E=\operatorname{St}(\theta)\vert_{I_F}$ and $\chi:=\theta\vert_{I_F}$
          then $\chi=\chi^{-1}$, thus $\chi$ is at most quadratic,
          \begin{equation}\label{m for St}
              \tau_E =\chi\oplus\chi  \quad \text{and} \quad m(\tau_E)=2m(\chi).
          \end{equation}

    \item If $\tau_E=(\Ind_{W_K}^{W_F} \theta)\vert_{I_F}$ and $\chi:=\theta\vert_{I_F}$
          then $\chi^{s}=\chi^{-1}$ and $\chi^{A}\vert_{\cl
              O_F^{\times}}=\varepsilon_K^{A}$. Moreover,
          \begin{equation}\label{m for SCU}
              \tau_E=\chi\oplus\chi^{-1} \quad \text{and} \quad m(\tau_E)=2m(\chi)~\text{if}~K/F~\text{is unramified},
          \end{equation}
          whilst
          \begin{equation}\label{m for SCR}
              \tau_E=\Ind_{I_K}^{I_F} \chi \quad \text{and} \quad m(\tau_E)=m(\chi)+m(\varepsilon_K)~\text{if}~K/F~\text{is ramified}.
          \end{equation}
    \item $\tau$ is the restriction of an exceptional supercuspidal representation (only occurs for $p=2$).
\end{enumerate}

Note that a single character of inertia $\chi$ is enough to determine $\tau$ in
all cases. This motivates the following definitions:

\begin{defn}\label{D: Associated inertia characters}
    Let $\tau$ be a principal series, Steinberg, ramified supercuspidal or unramified supercuspidal inertial type with trivial determinant. We say that $\chi$ is an inertia character associated to $\tau$ if it satisfies the equation~\eqref{m for PS}--\eqref{m for SCR}, respectively. For an exceptional inertial type $\tau$, we say that $\chi$ is an inertia character associated to $\tau$ if $\chi$ is associated to $\tau_L$ where $L$ is the minimal extension over which $\tau$ becomes triply imprimitive (see Proposition~\ref{P: base change of exceptional types}).
\end{defn}

\begin{remark}
    There are various characters associated to an inertial type. For instance, if $\tau$ is a principal series and $\chi$ is an associated inertia character so is $\chi^{-1}$. Also, if $\tau$ is triply imprimitive there are three extensions $K_i$ each one with two associated inertia characters $\chi_i$ and $\chi_i^{s_i}=\chi_i^{-1}$ so there is a total of $6$ associated inertia characters associated to~$\tau$.
\end{remark}

\subsection{Additional notation:}\label{Notation 2} When dealing with inertial types with unramified determinant, most of the time
we only care about $\theta\vert_{I_F}$ and not about $\theta$. To avoid
confusion, from now on we fix for the remain of the paper: the letters $\rho$
and $\theta$ will be reserved for representations and characters of $W_F$ while
$\tau:=\rho\vert_{I_F}$ and $\chi:=\theta\vert_{I_F}$ will be used for the
respective restrictions to inertia. Moreover, for a representation $\rho$ of
$W_F$ and an extension $L/F$ we set $\rho_L:=\rho\vert_{W_L}$ and
$\tau_L:=\rho_L\vert_{I_L}$.

\subsection{Elliptic curves}

Let $E/F$ be an elliptic with conductor~$N_E$. For each prime $\ell$, the
action of $\Gal(\overline{F}/F)$ on the $\ell^n$-torsion points of $E$ for all
$n>0$ induces an action of $\Gal(\overline{F}/F)$ on the Tate module
$T_{\ell}(E)=\varprojlim E[\ell^n]\simeq\ma Z_{\ell}^{2}$. After tensoring with
$\ma Q_{\ell}$ and choosing a basis for~$T_{\ell}(E)$ we get a representation
$\rho_{E,\ell}:\Gal(\overline{F}/F)\xrightarrow{}\GL_2(\ma Q_{\ell})$. When
$\ell\neq p$ one can construct from $\rho_{E,\ell}$ a Weil--Deligne
representation $(\rho_E,N)$, which is independent of the prime $\ell$. We
define the inertial type of $E$ to be $\tau_E=\rho_E\vert_{I_K}$.

It is well known that for an elliptic curve $E/F$ with potentially good
reduction the field $L:=F^{un}(E[m])$ is the same for all $m\in \ma Z_{\geq3}$
coprime to $p$. Moreover, $L$ is the fixed field of $\ker \tau_E$ and it is
called the {\it inertial field of $E$}. The order~$e_E$ of the group
$\Phi:=\Gal(L/F^{un})$ is the {\it semistability defect} of~$E$; the following
is standard (see~\cite{Kraus1990}).

\begin{theorem}\label{semistability groups}
    Let $E/F$ be an elliptic curve, and $L$ its inertial field. Then exactly one of the following holds

    \begin{enumerate}[(i)]
        \item $\Phi$ is cyclic of order $2,3,4,6$,
        \item $p=3$ and $\Phi\simeq C_3\rtimes C_4$ of order $12$,
        \item $p=2$ and $\Phi\simeq Q_8$ the quaternion group of order $8$, or
        \item $p=2$ and $\Phi\simeq \SL_2(\ma F_3)$ of order $24$.
    \end{enumerate}
\end{theorem}

\section{Some properties
  of types and their characters}\label{S: section 3}

In this section, we present some results that are used repeatedly in the following sections. Of special relevance, both for its utility and as a result on its own is Theorem~\ref{Inertia field determine the type}.

\begin{lemma}\label{L: Q8 has only one rep}
    Let $G$ be a group and $\rho_1, \rho_2 : G\xrightarrow{} \GL_2(\ma C)$ be two representations. Suppose that
    $\Image\rho_i\simeq Q_8$. Then $\rho_1\simeq\rho_2$ if and only if $\ker\rho_1=\ker\rho_2$. Moreover, the $\rho_i$ are irreducible.
\end{lemma}
\begin{proof}
    It is clear that $\rho_1\simeq\rho_2$ implies $\ker\rho_1=\ker\rho_2$. For the converse, notice that if $H:=\ker\rho_1=\ker\rho_2$ then both $\rho_i$ factor as faithful representations of the same group $G/H\simeq Q_8$. Since there is just one complex faithful representation of $Q_8$, it follows that $\rho_1\simeq\rho_2$. Moreover, this unique representation is irreducible.
\end{proof}

We denote by \href{https://www.lmfdb.org/Groups/Abstract/48.28}{$2O$} the
binary octahedral group, more precisely, $2O$ is the {\tt SmallGroup(48,28)}.
Recall that $F/\Q_p$ is a finite extension.

The following two lemmas are new
\begin{lemma}\label{L: if SL23 coincide on Q8 are equal}
    Let $L/F$ satisfy $I_F/I_L\simeq C_3$ and $\tau':I_L\xrightarrow{}\GL_2(\ma C)$ be an irreducible inertial type with $\det\tau'=1$. Then, there is at most one inertial type $\tau:I_F\xrightarrow{}\GL_2(\ma C)$ such that $\det\tau=1$ and $\tau\vert_{I_L}=\tau'$.
\end{lemma}
\begin{proof}
    Let $\tau_1$ and $\tau_2$ be two inertial types such that $\det\tau_i=1$ and $\tau_i\vert_{I_L}=\tau'$. Since $I_F/I_L\simeq C_3$, by Gallagher's Theorem~\cite[Corollary 6.17]{Isaacs1976}, $\tau_1\simeq\tau_2\otimes\delta$ for $\delta$ a character of $I_F$ with kernel~$I_L$. So, $\delta$ has order dividing~3, and taking determinants in the previous isomorphism yields $\det\tau_1\simeq\det\tau_2\tenso\delta^2$. Thus $\delta = 1$ and the result follows.
\end{proof}

\begin{lemma}\label{L:Exceptionals have unramified det}
    Let $G\in\{\GL_2(\F_3),2O\}$. Let $\rho:W_F\xrightarrow{}\GL_2(\F_3)$ be a representation with inertial type~$\tau$ satisfying $\Image\rho\simeq G$
    and $\Image\tau\simeq\SL_2(\F_3)$. Then $\det\tau=1$.
\end{lemma}
\begin{proof}
    Let $N$ be the field cut out by $\rho$. Then $\Gal(N/F)\simeq G$. The group $G$ admits a unique nontrivial cyclic quotient which is $G/\SL_2(\F_3)\simeq C_2$. The character $\det\rho$ cuts a cyclic subextension of $N/F$, so it must be either $F$ itself or the fixed field of $\SL_2(\F_3)$. In both cases, $\det\rho$ is unramified or equivalently $\det\tau=1$.
\end{proof}

The following theorem shows that an inertial type whose image is compatible
with that of a type arising from an elliptic curve is completely determined by
its inertial field.

\begin{theorem}\label{Inertia field determine the type}
    Let $\tau_1$ and $\tau_2$ be inertial types defined over $F$ with trivial determinant and image isomorphic to one of the groups in Theorem~\ref{semistability groups}. Let $N_i$ be the field fixed by $\ker \tau_i$. Then $N_1=N_2\iff\tau_1\simeq\tau_2$.
\end{theorem}

\begin{proof}

    Clearly, $\tau_1\simeq\tau_2\implies N_1=N_2$. For the converse, let $N' :=
        N_1=N_2$. Observe that this field is also the field fixed by the kernel of the
    associated inertia characters.

    Suppose first that $N'/F^{un}$ is abelian. Then, $\tau_i =
        \chi_i\oplus\chi_i^{-1}$ where $\chi_i$ is a character of $I_F$ if~$\tau_i$ is
    a principal series or $\chi_i$ is a character of $I_K=I_F$ if $\tau_i$ is
    unramified supercuspidal, where $K/F$ is the quadratic unramified extension.
    From $N_1=N_2$, we get $\ker\chi_1=\ker\chi_2$, and so both~$\chi_i$ factor
    through~$I_F/\ker\chi_i\simeq \Image\chi_i\simeq C_e$ and have image~$C_e$; so
    they are injective characters of $\Gal(N'/F^{un})=C_e$ with $e \in
        \{2,3,4,6\}$. Thus the image of a generator $g$ must be
    $\chi_i(g)=\zeta_e^{\pm1}$ and it follows that $\chi_1=\chi_2$ or
    $\chi_1=\chi_2^{-1}$ giving $\tau_1\simeq\tau_2$ in both cases.

    Suppose now that $N'/F^{un}$ is non-abelian and $p=3$. Then
    $\Gal(N'/F^{un})=C_3\rtimes C_4$ and $\tau_i=\Ind_{I_K}^{I_F} \chi_i$ for
    $\chi_i$ characters of $I_K=\Gal(\overline{K}/K^{un})$, where $K/F$ is
    quadratic ramified. The field fixed by~$\ker\tau_i$ is equal to the field fixed
    by~$\ker \chi_i$. Since $K^{un}$ is the unique quadratic ramified extension of
    $F^{un}$, it follows that $\chi_i$ are injective characters of
    $\Gal(N'/K^{un})\simeq C_6$ and, by the same argument as before,
    $\chi_1=\chi_2$ or $\chi_2^s$ thus $\tau_1\simeq\tau_2$.

    Finally, assume that $N'/F^{un}$ is non-abelian and $p=2$. Then
    $\Gal(N'/F^{un})=Q_8$ or $\SL_2(\ma F_3)$. The former case follows from
    Lemma~\ref{L: Q8 has only one rep}. For the latter, notice that if $L$ is the
    unique cubic (tame) extension of $F^{un}$, then $\ker \tau_1\vert_{I_L} =\ker
        \tau_2\vert_{I_L} $ and $\Image \tau_i\vert_{I_L} \simeq Q_8$. By the previous
    case, we have $\tau_1\vert_{I_L}\simeq\tau_2\vert_{I_L}$ and the result follows
    from Lemma~\ref{L: if SL23 coincide on Q8 are equal}.
\end{proof}

\begin{lemma}\label{L: Inductions have image Q8}
    Let $G$ be a group, $H\leq G$ an index 2 subgroup, and consider the irreducible induction $\rho:=\Ind_H^G \theta$ where $\theta$ is a character of order 4 or 6. Assume that $\rho$ has trivial determinant. Then $\Image\rho$ is equal to the quaternion group~$Q_8$ if and only if $\theta$ has order 4 and $\Image\rho\simeq C_3\rtimes C_4$ if and only if $\theta$ has order 6.
\end{lemma}

\begin{proof}
    Since $\rho$ has trivial determinant its image is a finite subgroup of $\SL_2(\ma C)$ whose classification is well known. There are two infinite families: the cyclic groups $C_n$ and the binary dihedral or dicyclic groups $\operatorname{Dic}_n$ of order $4n$ and three exceptional groups, of order $24,48$ and $120$. Since the image is non-abelian of order 8 or 12 the only possibilities are $\operatorname{Dic}_2\simeq Q_8$ and $\operatorname{Dic}_3\simeq C_3\rtimes C_4$
\end{proof}

\begin{lemma}\label{L: Elliptitc Curve SCR conditions for chi}
    Let $\tau=\Ind_{I_K}^{I_F}\chi$ be a ramified supercuspidal inertial type. Then, $\tau$ has trivial determinant and $\Image\tau=Q_8$ (respectively $\Image\tau=C_3\rtimes C_4$) if and only if $\chi^A\vert_{\cl O_F^{\times}}=\varepsilon_K^{A}$ and the order of $\chi$ is $4$ (respectively 6)
\end{lemma}
\begin{proof}
    This follows directly form Lemma~\ref{cyclotomicdeterminant} and Lemma~\ref{L: Inductions have image Q8}.
\end{proof}

\begin{lemma}\label{L: explicit tiply conditions}
    Let $\tau=\Ind_{I_K}^{I_F}\chi$ be a ramified supercuspidal inertial type and let $\pi\in K$ be a uniformizer. Assume that $\chi$ has order $4$ and $\chi\vert_{\cl O_F^{\times}}=\varepsilon_K^{A}$. Then, the type $\tau$ is triply imprimitive if and only if $\chi(\pi/s(\pi))=\pm 1$ and it is simply imprimitive if and only if $\chi(\pi/s(\pi))=\pm i$.
\end{lemma}
\begin{proof}
    Let $\theta$ be any character of $K^{\times}$ such that $\theta\vert_{\cl O_K^{\times}}=\chi$. By Proposition~\ref{triply imprimitive}, we have that $\rho=\Ind_{W_K}^{W_F}\theta$ is triply imprimitive if and only if $\theta/\theta^s$ factors through the norm map, and simply imprimitive otherwise. By Lemma~\ref{Factor through the norm}, this is equivalent to $(\theta/\theta^s)^s=\theta/\theta^s$, or equivalently, $\theta/\theta^s$ having order dividing $2$. By Lemma~\ref{cyclotomicdeterminant}, $\chi^s=\chi^{-1}$, so $(\theta/\theta^s)\vert_{\cl O_K^{\times}}=\chi^2$. Since $\chi$ has order $4$ then $\theta/\theta^s$ is quadratic on $\cl O_K^{\times}$. Thus it will be quadratic on $K^{\times}$ if and only if $(\theta^s/\theta)(\pi)=\pm1$ for any choice of uniformizer $\pi$. Since $\pi/s(\pi)$ is a unit, we have $(\theta/\theta^s)(\pi)=\frac{\theta(\pi)}{\theta(s(\pi))}=\theta(\frac{\pi}{s(\pi)})=\chi(\pi/s(\pi))$ and the result follows.
\end{proof}

\section{Uniform Cases}\label{S: Uniform cases}

In this section, we will present the cases that can be treated uniformly.

\begin{lemma}\label{L: Quadratic Twists of Ell Curves} Let $E/F$ be an elliptic curve with semistability defect $e \in \{2,6\}$.
    If $e=2$, then~$E$ has a quadratic twist $E_d$ with good reduction, and if $e=6$ then it has a quadratic twist $E_d$ with semistability defect $3$.
\end{lemma}
\begin{proof}
    For $p\geq 3$ it is well known. For $p=2$ see~\cite[Lemma 4.2]{NirvannaWildCyclic}.
\end{proof}

Before showing the classification of types we need to define appropriate labels
for them.

\begin{defn}\label{Labels of types}
    We define $\tau_{sc}(u,m,e)$ to be a supercuspidal type induced from
    $K=F(\sqrt{u})$ by a character $\theta$ of $W_K$ with $m(\theta)=m$ and restriction to~$I_K$ of order~$e$. Similarly, we denote by $\tau_{ps}(1,m,e)$ a principal series type associated to a character  $\theta$ of $W_F$ with $m(\theta)=m$ and restriction to~$I_F$ of order~$e$. Finally, $\tau_{St}$ denotes the inertial type corresponding to~$[St(1),N]$.
\end{defn}

\begin{lemma}\label{Pot Good Uniform}
    Let $E/F$ be an elliptic curve with semistability defect $e\leq 2$. Then,
    \begin{enumerate}[(i)]
        \item If $e=1$, that is, $E$ has good reduction, then $v(N_E)=0$ and
              $\tau_E=1\oplus1$;
        \item If $p\geq 3$ and $e=2$, then $v(N_E)=2$ and $\tau_E=\varepsilon_K\oplus
                  \varepsilon_K$ where $K$ is any of the two quadratic ramified extensions of
              $F$;
        \item If $p=2$ and $e=2$, then $v(N_E)=2m(\varepsilon_K)$ and
              $\tau_E=\varepsilon_K\oplus\varepsilon_K$ for some quadratic~$K/F$.
    \end{enumerate}
\end{lemma}

\begin{proof}
    The proof is analogous to the proof of~\cite[Lemma 3.2.4]{dembélé2024galoisinertialtypeselliptic}
\end{proof}

Recall that $\tau_{St}$ denotes the equivalence class $[St(1),\left(\begin{smallmatrix} 0&1\\0&0\end{smallmatrix}\right)]$ and we will denote by $\varepsilon_i\tenso\tau_{St}$ the equivalent class $[St(\varepsilon_i),\left(\begin{smallmatrix} 0&1\\0&0\end{smallmatrix}\right)]$.

\begin{lemma}\label{Pot Mult Uniform}
    Let $E/F$ be an elliptic curve $v(j_E) < 0$, that is, with potentially multiplicative reduction. Then,
    \begin{enumerate}[(i)]
        \item If $E$ has multiplicative reduction, then $N_E=\p$ and
              $\tau_E=\tau_{\operatorname{St}}$.
        \item If $p\geq 3$ and $E$ has additive potentially multiplicative reduction, then
              $N_E=\p^2$ and $\tau_E=\varepsilon_K\tenso\tau_{\operatorname{St}}$ for any of
              the two quadratic unramified extensions $K/F$.
        \item If $p=2$ and $E$ has additive potentially multiplicative reduction, then
              $\nu(N_E)=2m(\varepsilon_K)$ and $\tau_E=\varepsilon_K \tenso
                  \tau_{\operatorname{St}}$ for some quadratic extension $K/F$.
    \end{enumerate}
\end{lemma}
\begin{proof}
    The proof is analogous to the proof of~\cite[4.1.1]{dembélé2024galoisinertialtypeselliptic}
\end{proof}

\subsection{Inertia types for \texorpdfstring{$p \geq 5$}{}}

Recall that $F/\Q_p$ is a finite extension and assume $p \geq 5$ for this
section. We will classify the inertial types $\tau_E$ of elliptic curves $E/F$.
The cases of potentially multiplicative reduction and potentially good
reduction with $e\leq2$ were treated in the previous section, therefore it
remains to study the case of curves with potentially good reduction and
$e\in\{3,4,6\}$.

Let $u\in \cl O_F^{\times}$ generate $(\calO_F/\fp)^\times \simeq
    \F_q^{\times}$.

\begin{lemma}\label{Pot Good p=5}
    Let $E/F$ be an elliptic curve with potentially good reduction and semistability defect $e\geq 3$.
    We have
    \[\tau_E=\begin{cases}  \tau_{ps}(1,1,e) \quad \text{if} \quad  e\vert(q-1); \\
            \tau_{sc}(u,1,e)\quad \text{if} \quad e\vert(q+1),\end{cases}\]
    where $\tau_{ps}$ and~$\tau_{sc}$ are given by Definition~\ref{Labels of
        types}.
\end{lemma}
\begin{proof}
    The proof is perfectly analogous to the proof of \cite[Proposition 4.2.1]{dembélé2024galoisinertialtypeselliptic}
\end{proof}

\begin{corollary}\label{explicit inertial types for p>5}
    Let $E/F$ be an elliptic curve with semistability defect $e\in \{3,4,6\}$.

    1) If $n$ is even, then $\tau_{E}$ is either $\tau_{ps}(1,1,4),\tau_{ps}(1,1,3)$ or $\tau_{ps}(1,1,6)$.

    2) If $n$ is odd, then we have the following cases:
    \begin{itemize}
        \item If $p\equiv1 \pmod{12}$, then $\tau_{E}$ is either
              $\tau_{ps}(1,1,4),\tau_{ps}(1,1,3)$ or $\tau_{ps}(1,1,6)$;
        \item If $p\equiv5 \pmod{12}$, then $\tau_{E}$ is either
              $\tau_{ps}(1,1,4),\tau_{sc}(u,1,3)$ or $\tau_{sc}(u,1,6)$;
        \item If $p\equiv7 \pmod{12}$, then $\tau_{E}$ is either
              $\tau_{sc}(u,1,4),\tau_{ps}(1,1,3)$ or $\tau_{ps}(1,1,6)$;
        \item If $p\equiv11 \pmod{12}$, then $\tau_{E}$ is either
              $\tau_{sc}(u,1,4),\tau_{sc}(u,1,3)$ or $\tau_{sc}(u,1,6)$.
    \end{itemize}
\end{corollary}
\begin{proof}
    Since $p\geq5$, we have $p\equiv \pm 1 \pmod 6$ and $p\equiv \pm1 \pmod4$.

    If $n$ is odd, then $q=p^n\equiv p\pmod6$ and $q=p^n\equiv p\pmod4$. By the
    Chinese remainder theorem, $q\equiv p\pmod{12}$. If $n$ is even then $p^n\equiv
        1 \pmod{12}$. In both cases, the result follows by Proposition~\ref{Pot Good
        p=5}.
\end{proof}

This corollary together with Lemmas~\ref{Pot Good Uniform} and~\ref{Pot Mult
    Uniform} provide a complete classification of the inertial types of elliptic
curves $E$ defined over any finite extension $F/\ma Q_p$ for $p\geq 5$.

\begin{example}\label{Example p=5} Let $\pi$ be a uniformizer of $F$. An application of~\cite[\S 1, Proposition 1]{Kraus1990} shows that the following curves have the following inertial types.
    \begin{enumerate}
        \item The curve $y^2=x^3+\pi^2x$ has inertial type $\tau_{ps}(1,1,2)$;
        \item The curve $y^2=x^3+\pi^2$ has inertial type $\tau_{ps}(1,1,3)$ or
              $\tau_{sc}(u,1,3)$;
        \item The curve $y^2=x^3+\pi x$ has inertial type $\tau_{ps}(1,1,4)$ or
              $\tau_{sc}(u,1,4)$;
        \item The curve $y^2=x^3+\pi$ has inertial type $\tau_{ps}(1,1,6)$ or
              $\tau_{sc}(u,1,6)$,
    \end{enumerate}
    where in (2)--(4), the type is principal series or supercuspidal depending on $p$ and~$n$ as prescribed by Corollary~\ref{explicit inertial types for p>5}.
\end{example}

\subsection{Tame cases for $p=2,3$}
The tame cases for $p=2$ and $p=3$ can also be treated uniformly. Let $\pi \in
    F$ be a uniformizer, and consider the elliptic curves
\[
    E_1 / F : \; y^2+\pi y=x^3 \qquad \text{ and } \qquad E_2 / F \; : \; y^2=x^3-\pi x,
\]
where $F$ is a finite extension of $\Q_2$ for $E_1$ and a finite extension of
$\Q_3$ for $E_2$. An application of Tate's algorithm and~\cite{Kraus1990} shows
that their semistability defects are $e(E_1) = 3$ and $e(E_2) = 4$; in
particular, they have tame reduction and $v(N_{E_i})=2$.
\begin{lemma}\label{L: tame for p=2 and p=3}
    Let $F/\Q_2$ (respectively $F/\Q_3$) be a finite extension.
    There is a unique inertial type $\tau$ defined over $F$ with trivial determinant and image $C_3$ (respectively $C_4$). Moreover, it satisfies $m(\tau) = 2$ and
    \begin{enumerate}[(a)]
        \item $\tau$ is a principal series if $\mu_3\subset F$ (respectively $\mu_4\subset F$) and unramified supercuspidal otherwise;
        \item $\tau$ is the inertial type of the elliptic curve $E_1 / F$ (respectively $E_2 / F$).
    \end{enumerate}
\end{lemma}

\begin{proof}

    Let $\tau$ be a inertial type defined over a finite extension $F/\Q_2$ with
    trivial determinant and image isomorphic to~$C_3$. Since $\Image \tau \simeq
        C_3$, it follows from~\S\ref{sec:typesCurves} that $\tau$ is either a principal
    series or unramified supercuspidal. In both cases, $\tau\simeq \chi\oplus
        \chi^{-1}$ for $\chi:I_F\xrightarrow{} \ma C$ a character of order $3$. Since
    $\gcd(2,3)=1$, there is a unique $C_3$-extension~$N/F^{un}$, fixed by the
    kernel of a character $\varepsilon_N$. It follows that $\chi$ is equal to
    $\varepsilon_N^{\pm 1}$ with both possibilities giving the same inertial type,
    which is then unique, and satisfies $m(\tau) = 2$. This proves the first part
    of the lemma.

    If $\mu_3\subset F$, then $\mu_3\subset(\cl O_F/\p)^{\times}$ and so there is
    an order $3$ character $\chi^A:(\cl O_F/\p)^{\times}\xrightarrow{}\ma C$. Let
    $\theta$ be a character of ~$W_F$ such that $\theta|_{I_F} = \chi$. It follows
    that $\tilde{\tau}:= PS (\theta, \theta^{-1})|_{I_F}=\chi\oplus\chi^{-1}$ is a
    principal series with trivial determinant, image $C_3$ and it satisfies
    $m(\tilde{\tau})=2$.

    If $\mu_3\not \subset F$, then $\mu_3\not\subset(\cl O_F/\p)^{\times}$ and
    $\mu_3\subset(\cl O_K/\p)^{\times}$ where $K/F$ is quadratic unramified, and so
    there is a character $\chi^A$ of $(\cl O_K/\p)^{\times}$ of order $3$.
    Moreover, $(\cl O_F/\p)^{\times}$ is cyclic of order coprime to $3$, so we can
    choose $\chi^A$ to be trivial on $(\cl O_F/\p)^{\times}$. Let $\theta$ be a
    character of~$W_K$ such that~$\theta|_{I_K} = \chi$. Since $K/F$ is unramified,
    the character $\varepsilon_K$ is trivial and we have $\theta^{A}\vert_{\cl
        O_F^{\times}}=\chi^{A}\vert_{\cl O_F^{\times}}=\varepsilon_K^{A} = 1$. By
    Lemma~\ref{cyclotomicdeterminant}, we have $\chi^s=\chi^{-1} \neq \chi$. It
    follows that $\tilde{\tau}:= (\Ind_{W_K}^{W_F} \theta)|_{I_F}
        =\chi\oplus\chi^s$ is unramified supercuspidal with trivial determinant, image
    $C_3$ and it satisfies $m(\tilde{\tau})=2$.

    By the uniqueness of $\tau$ we conclude $\tilde{\tau} = \tau$ in both cases,
    proving (a). Part (b) now follows directly from the paragraph preceding the
    lemma.

    The case of $F$ a finite extension of $\Q_3$ is analogous.
\end{proof}

\section{Necessary conditions}\label{S: Necessary conditions}

In this section, we show that the hypothesis appearing in Theorem~\ref{T: Main
    Theorem for p=3} and Theorem~\ref{T: Main Theorem for p=2} are satisfied by
types arising from elliptic curves.

\subsection{Necessary conditions for $p=3$}
The proof that the assumptions in Theorem~\ref{T: Main Theorem for p=3} are
necessary is straightforward. Indeed let $F/\Q_3$ be a finite extension and
$E/F$ an elliptic curve. Let $\tau$ be its inertial type with associated
inertia character $\chi$. It is clear that $\det \tau =1$. Assume first that
$\tau$ is a principal series. Since the image of inertia is abelian, by
Theorem~\ref{semistability groups} it follows that the order of $\chi$ is
$2,3,4,6$. If $\tau_E$ is unramified supercuspidal the same reasoning applies,
excluding $e=2$ as this only occurs for principal series by Lemma~\ref{Pot Good
    Uniform}; thus $\chi$ has order $3,4$ or $6$. If $\tau_E$ is ramified
supercuspidal, then its image is irreducible and so by
Theorem~\ref{semistability groups} it must be $C_3\rtimes C_4$.

\subsection{Necessary conditions for $p=2$}

We will now prove that the hypothesis of Theorem~\ref{T: Main Theorem for p=2}
are necessary. The proof of $(i)$ and $(ii)$ is completely analogous to the
case $p=3$, but proving $(iii)$ and $(iv)$ requires considerably more effort.

The following theorem implies that the hypothesis of Theorem~\ref{T: Main
    Theorem for p=2} part $(iii)$ are necessary.

\begin{theorem}\label{T: Images of Galois}
    Let $E/F$ an elliptic curve with potentially good reduction. Let $\rho_E$ be the Weil-Deligne representation associated with $E$ and $\tau=\rho_E\vert_{I_F}$. Let $G=\operatorname{Gal(F(E[3])/F)}$. Suppose that $\tau(I_F)$ is nonabelian.
    \begin{enumerate}[(a)]
        \item If $\mu_3 \subset F$, then we have the following possibilities:
              \begin{enumerate}[(i)]
                  \item $\rho_E$ is triply imprimitive with $\tau(I_F)\simeq Q_8$ and $G= Q_8$;
                  \item $\rho_E$ is exceptional with $\tau(I_F)\simeq Q_8$ and $G=\SL_2(\ma F_3)$, or
                  \item $\rho_E$ is exceptional with $\tau(I_F)\simeq \SL_2(\ma F_3)$
                        and $G=\SL_2(\ma F_3)$.
              \end{enumerate}

        \item If $\mu_3 \not\subset F$, then we have the following possibilities:
              \begin{enumerate}[(i)]
                  \item $\rho_E$ is simply imprimitive with $\tau(I_F)\simeq Q_8$ and $G=\SD_{16}$, the $2$-sylow subgroup of~$\GL_2(\ma F_3)$, or
                  \item $\rho_E$ is exceptional with $\tau(I_F)\simeq\SL_2(\ma F_3)$
                        and $G=\GL_2(\ma F_3)$.
              \end{enumerate}
    \end{enumerate}
\end{theorem}
\begin{proof}
    Let $\rho:=\rho_E$ and let $\rhobar_{E,3}$ be the 3-torsion representation of $E$. The inertial types with nonabelian image of inertia are either exceptional or ramified supercuspidal types.

    By~\cite[Lemma 1]{Dokchitser_2008} there is an unramified twist
    $\eta\tenso\rho$ such that $\Image(\eta\tenso\rho)\simeq G$, note that an
    unramified twist does not affect the inertial type. We also have
    $\Image\overline{\rho}_{E,3}\simeq G$ as abstract groups and, moreover, if $G$
    is non-abelian, then the projective images of $\eta\tenso\rho$ and
    $\overline{\rho}_{E,3}$ are both isomorphic to~$G/Z(G)$. Since twisting does
    not affect the projective image, we conclude that the projective images of
    $\rho$ and~$\rhobar_{E,3}$ are isomorphic to $G/Z(G)$ as abstract groups. We
    know that $\tau_E(I_F)$ must be isomorphic to a subgroup of $G$, below we use
    the table in~\cite[Proposition 2]{Dokchitser_2008} to distinguish between
    cases.

    Assume $\mu_3\subset F$. Then, the possibilities for~$G$ containing
    $\tau(I_F)=Q_8$ are: $G=Q_8$, case (i), $G=\SL_2(\F_3)$, case (ii), and for $G$
    containing $\tau(I_F)\simeq\SL_2(\ma F_3)$ is $G=\SL_2(\ma F_3)$, case (iii).
    To complete the proof of (a), it remains to show which inertial type correspond
    to each case. The projective image of $\overline{\rho}_{E,3}$ is either $D_2$
    in case (i) or $A_4$ in cases (ii) and (iii). By Proposition~\ref{triply
        imprimitive} the first case implies that $\rho_E$ is triply imprimitive and the
    last two imply that $\tau_E$ is exceptional.

    Assume $\mu_3\not\subset F$. Then, for $\tau(I_F)=Q_8$ the possible $G$ are
    $\SD_{16}$ and~$\GL_2(\ma F_3)$; the former has projective image $D_4$ so, by
    Proposition~\ref{triply imprimitive}, it must be simply imprimitive; the latter
    has projective image $S_4$ so it is exceptional. For $\tau(I_F)=\SL_2(\ma F_3)$
    the only option is $G = \GL_2(\ma F_3)$.

    To complete the proof of (b), we need to show that if $\tau$ is exceptional,
    then $\tau(I_F)=Q_8$ is not possible. Assume that $\rho_E$ is exceptional.
    Then, $G=\GL_2(\ma F_3)$ and the projective image of $\overline{\rho}_{E,3}$ is
    $S_4$. By the reasoning in the first paragraph of \cite[\S 42.3]{Bushnell2006},
    projective image~$S_4$ corresponds to the ``Octaedral'' case, that is, an
    exceptional representation that becomes imprimitive over a non Galois cubic
    extension $C/F$, and becomes triply imprimitive over $L$, the Galois closure of
    $C/F$. If $\tau_L$ is triply imprimitive in particular $\tau(I_L)=Q_8$ so the
    only possibility to get $\tau(I_F)=Q_8$ is that $C/F$ is unramified and non
    Galois which is impossible. Therefore, if $G=\GL_2(\ma F_3)$ necessarily the
    image of inertia is $\SL_2(\ma F_3)$.
\end{proof}

\begin{remark}\label{R: Q8conductor}
    Observe that there are inertial types with image $Q_8$ and conductor possible for elliptic curves that do not arise from elliptic curves. For instance, the last bullet point of the proof of \cite[Proposition 6.5.1]{dembélé2024galoisinertialtypeselliptic} exhibits a triply imprimitive type with image $Q_8$ and conductor $2^8$ defined over $\ma Q_2$. Since $\mu_3\not\subset \Q_2$ it follows form the previous theorem that this type does not come from an elliptic curve over $\Q_2$.
\end{remark}

\begin{proposition}\label{P: Conditions for Exceptionals triply}
    Let $E/F$ then $\tau_E$ satisfies $(a),(b),(c),(d)$ in part $(iv)$ of Theorem~\ref{T: Main Theorem for p=2}.
\end{proposition}

\begin{figure}[h]
    \[
        \xymatrix{
            & F(E[3])  \ar@{-}[d]^2 &     &&
            & F(E[3])  \ar@{-}[d]^2\\
            & M \ar@{-}[d]^2 \ar@{-}[dl] \ar@{-}[dr]  &     &&
            & M \ar@{-}[d]^2 \ar@{-}[dl] \ar@{-}[dr] \\
            K_1 \ar@{-}[dr]& K_2 \ar@{-}[d]^2 & K_3 \ar@{-}[dl] &&
            K_1 \ar@{-}[dr] & K_2 \ar@{-}[d]^2 & K_3 \ar@{-}[dl] \\
            & L \ar@{-}[d]^3   &     &&     & L \ar@{-}[dl]_2 \ar@{-}[dr]^3 \\
            & F    &     && C \ar@{-}[dr]_3   &     & F(\mu_3) \ar@{-}[dl]^2 \\
            &    &     &&     & F
        }
    \]
    \caption{The left diagram corresponds to an exceptional type when $\mu_3\subset F$; the right diagram to the case when $F$ does not contain $\mu_3$}.
    \label{fig:exceptional diagram}
\end{figure}
\begin{proof}
    The strategy is to base change the exceptional
    inertial types to the field $L$ where they become triply imprimitive. By~\cite[Lemma 1]{Dokchitser_2008} we can consider $\tilde{\rho}=\rho_E\tenso\eta$ where $\eta$ is an unramified character and $\tilde{\rho}$ factors through $\Gal(F(E[3])/F)$. Thus by Theorem~\ref{T: Images of Galois} $\Image\tilde{\rho}$ is $\SL_2(\F_3)$ if $\mu_3\subset F$ or $\GL_2(\F_3)$ if $\mu_3\not\subset F$. In both cases we can consider $L$ to be the fixed field of the subgroup $Q_8$ inside $\Image\tilde{\rho}$ and $\rho:=\tilde{\rho}_L$. It is clear that $\Image\rho=\Image\tau_L=Q_8$ and thus $\rho$ has projective image $Q_8/Z(Q_8)=D_2$ so by Proposition~\ref{triply imprimitive} $\rho$ is triply imprimitive. By looking at the lattice of subgroups of $\Image\tilde{\rho}$ it follows that the field $L$ is precisely the field of Proposition~\ref{P: base change of exceptional types}, more precisely it is $L$ if $\mu_3\subset F$ or the $\tilde{L}$ if $\mu_3\not\subset F$. Thus we know that the fields from where $\tilde{\rho}$ is induce are permuted transitively and thus they form an irreducible submodule of $L^{\times}/(L^{\times})^2$. The Galois invariance of $\rho$ comes from the fact that $\rho$ is the restriction of $\tilde{\rho}$ define over $W_F$.
\end{proof}

\section{Sufficient conditions}\label{S: sufficient conditions}

This section is devoted to proof the remaining direction of the equivalences in
our main theorems. Given a $\tau$ as in their statements, our strategy is to
build a continuous representation $\rho$ and an injection $\iota:\Image \rho
    \hookrightarrow{}\GL_2(\F_{\ell})$ in such a way that we can apply the
following proposition. Let $\chi_{\ell}$ be the$\mod\ell$ cyclotomic character.

\begin{proposition}\label{Rubin}
    Let $\rho: G_F \xrightarrow{}\ma C$ be a continuous representation and let $\tau:=\rho\vert_{I_F}$. Let $\iota:\Image \rho \hookrightarrow{}\GL_2(\F_{\ell})$ be an injection with $\ell=3$ or $5$ and such that $\det(\iota\circ\rho)=\chi_{\ell}$. Then, there is an elliptic curve $E/F$ with $\tau_E\simeq\tau$.
\end{proposition}

\begin{proof}
    Since $\iota\circ\rho$ has cyclotomic determinant, by~\cite[Theorem 3]{Rubin}, there is an elliptic curve $E/F$ with $\barra{\rho}_{E,\ell}\simeq\iota\circ\rho$. Since $\overline{F}^{\ker\tau_E}=F^{un}(E[\ell])$ restricting $\barra{\rho}_{E,\ell}\simeq\iota\circ\rho$ to inertia gives $$\overline{F}^{\ker \tau_E}=F^{un}(E[\ell])=\overline{F}^{\ker\tau}$$ and hence $\tau_E\simeq\tau$ by Theorem~\ref{Inertia field determine the type}.
\end{proof}

To apply our strategy, we need to control the image of $\rho$ which in the
ramified supercuspidal case will be given by the following lemma.

\begin{lemma}\label{L: Image of ramified inductions}
    Let $K/F$ be a quadratic ramified extension, $\theta$ a character of $W_K$, and fix $\pi\in \cl O_K$ a uniformizer. Set $\chi:=\theta\vert_{I_K}$ and $\lambda:=\chi^A(\pi/s(\pi))$. Let $\rho=\Ind_{W_K}^{W_F} \theta $. If $\operatorname{Ord}(\chi)=d\in\{4,6\}$ and $\chi^{A}\vert_{\cl O_F^{\times}}=\varepsilon_K^{A}$, then $\Image\rho\subset \GL_2(\ma C)$ is generated by
    \begin{align*}
        \rho(u_1)=\begin{pmatrix}
                      \zeta_d & 0            \\
                      0       & \zeta_d^{-1}
                  \end{pmatrix}
        \quad\quad\quad
        \rho(u_2)=\begin{pmatrix}
                      0       & -\zeta_d^{-1} \\
                      \zeta_d & 0
                  \end{pmatrix}
        \quad\quad\quad
        \rho(\Frob_K)=\begin{pmatrix}
                          \theta(\Frob_K) & 0                                 \\
                          0               & \lambda^{-1}\cdot \theta(\Frob_K)
                      \end{pmatrix}
    \end{align*}
    for some $u_1,u_2\in I_F$ where $\Frob_K=\Art_K(\pi)$. In particular, \begin{enumerate}[(a)]
        \item If $d=4$, $\lambda=1$ and $\theta(\Frob_K)=1$ then $\Image \rho =Q_8$;
        \item If $d=4$, $\lambda=-1$ and $\theta(\Frob_K)=i$ then $\Image \rho =Q_8$;
        \item If $d=4$, $\lambda=-i$ and $\theta(\Frob_K)=\zeta_8$ then $\Image \rho
                  =\SD_{16}$;
        \item If $d=6$, $\lambda=-1$ and $\theta(\Frob_K)=\zeta_8$ then $\Image \rho
                  =C_{24}\rtimes C_2$;
        \item If $d=4$, $\lambda=1$ and $\theta(\Frob_K)=i$ then $\Image \rho =C_4\times
                  S_3$;
        \item If $d=4$, $\lambda=1$ and $\theta(\Frob_K)=1$ then $\Image \rho =C_3\rtimes
                  C_4$.
    \end{enumerate}

\end{lemma}
\begin{proof}
    The image of $\rho$ is generated by $\rho(I_F)$ and $\rho(\Frob_F)$ for any $\Frob_F$. By Lemma~\ref{cor: the inductions are really irreducible}, we know that $\rho\vert_{I_F}=\Ind_{I_K}^{I_F} \chi$ is irreducible and $[I_K:I_F]=2$ so by Lemma~\ref{L: Inductions have image Q8}  $\rho(I_F)=Q_8$ if $d=4$ and $\rho(I_F)=C_3\rtimes C_4$ if $d=6$. In both cases one can easily check that $\rho(I_F)$ is generated by the image of two elements $u_1,u_2$ as in the statement. Since $K/F$ is quadratic ramified, any Frobenius element for $K$ is a Frobenius for $F$, so we can choose $\Frob_F:=\Frob_K=\Art(\pi)$. Finally, we have
    $\rho(\Frob_K)=\left( \begin{smallmatrix} \theta(\Frob_K)& 0 \\ 0& \theta^s(\Frob_K) \end{smallmatrix}\right)$ and, by hypothesis, $$\lambda=\chi^A(\pi/s(\pi))=\theta^A(\pi)/\theta^A(s(\pi))=\theta(\Frob_K)/\theta^s(\Frob_K)\implies \theta^s(\Frob_K)=\lambda^{-1}\theta(\Frob_K).$$

    The last statements follow from a direct computation.
\end{proof}

\begin{corollary}\label{C: Rho with Q8 exits}
    Let $\rho$ be a triply imprimitive representation of~$W_F$ with unramified determinant. If $\rho(I_F)=Q_8$, then there is a triply imprimitive representation $\rho'$ such that $\rho\vert_{I_F}=\rho'\vert_{I_F}$ and $\rho'(W_F)=Q_8$. Moreover, $\rho'=\eta\tenso\rho$ where $\eta$ is an unramified character of $W_F$.
\end{corollary}

\begin{proof}
    Let $\theta$ be the character inducing $\rho$ from a ramified quadratic extension $K/F$. Since $\rho(I_F)=Q_8$ and $\rho$ has unramified determinant, we have that $\chi:=\theta|_{I_K}$ has order~4 and
    Lemma~\ref{L: explicit tiply conditions} implies $\lambda = \chi^A(\pi/s(\pi))=\pm1$.
    Let $\eta$ be the unramified character satisfying $\eta(Frob_K)=\theta(Frob_K)^{-1}$ if $\lambda=1$ or $\eta(Frob_K)=-i\cdot\theta(Frob_K)$ if $\lambda=-1$. Then, $\theta' = \theta\otimes\eta$ satisfies Lemma~\ref{L: Image of ramified inductions} $(a)$ or $(b)$.
    Therefore, $\rho'=\Ind_{W_K}^{W_F} \theta' $ satisfies that
    $\rho\vert_{I_F}=\rho'\vert_{I_F}$ and $\rho'(W_F)=Q_8$. Moreover, since
    $K/F$ is totally ramified,
    $\eta$ extends to $W_F$ and so $\rho'=\eta\tenso\rho$.
\end{proof}

\begin{lemma}\label{L: Image of unramified inductions}
    Let $K/F$ be quadratic unramified, $\theta$ a character of $W_K$ and $\chi:=\theta|_{I_K}$. Define $\gamma:=\theta(\Frob_F^2)$ and $\rho=\Ind_{W_K}^{W_F}\theta$. If $\operatorname{Ord}(\chi)=d\in\{4,6\}$ and $\chi^{s}=\chi^{-1}$, then $\Image\rho\subset \GL_2(\ma C)$ is generated by
    \begin{align*}
        \rho(u_1)=\begin{pmatrix}
                      \zeta_d & 0            \\
                      0       & \zeta_d^{-1}
                  \end{pmatrix}
        \quad\quad\quad
        \rho(\Frob_F)=\begin{pmatrix}
                          0 & \gamma \\
                          1 & 0
                      \end{pmatrix}
    \end{align*}
    for some $u_1\in I_F$. In particular, \begin{enumerate}[(a)]
        \item If $d=3$ and $\gamma=1$ then $\Image \rho =C_3\rtimes C_4$;
        \item If $d=3$ and $\gamma=-1$ then $\Image\rho =S_3$;
        \item If $d=3$ and $\gamma=i$ then $\Image\rho=C_3\rtimes C_8$;
        \item If $d=4$ and $\gamma=1$ then $\Image\rho=D_4$;
        \item If $d=4$ and $\gamma=-1$ then $\Image\rho=Q_8$;
    \end{enumerate}
\end{lemma}

\begin{proof}
    The image of $\rho$ is generated by $\rho(I_F)$ and $\rho(\Frob_F)$. Since $K/F$ is unramified, we have that $I_F=I_K$ and $\Frob_F^2$ is a Frobenius for $K$. Thus, $\rho(I_F)=\rho(I_K)\simeq\theta(I_K)\simeq C_d$ is generated by a matrix as in the statement. Since $\{1,\Frob_F\}$ is a set of representatives for~$W_F/W_K$, using the definition of induced representation, it is easy to check that $\rho(\Frob_F)$ is the desired matrix. The last statements follows by direct computation.
\end{proof}

\subsection{Sufficient conditions for $p=3$} We will use $\ell=5$ to build the representation $\bar{\rho}:=\iota\circ\rho$
as in Proposition~\ref{Rubin}. We need the following auxiliary lemmas.

\begin{lemma}\label{L: injections p=3} Let $G$ be a group.
    The maps $\iota$ defined below on generators of $G$ is an isomorphism from $G$ as a subgroup of $\GL_2(\ma C)$ to $G$ as subgroup of $\GL_2(\ma F_5)$.
    \begin{enumerate}[(a)]
        \item $G=C_4\times S_3$
              \begin{align*}
                  \begin{pmatrix}
                      \zeta_6 & 0            \\
                      0       & \zeta_6^{-1}
                  \end{pmatrix}
                   & \xrightarrow{\iota}
                  \begin{pmatrix}
                      3 & 1 \\
                      3 & 3
                  \end{pmatrix}
                  \quad\quad\quad
                  \begin{pmatrix}
                      0       & -\zeta_6^{-1} \\
                      \zeta_6 & 0
                  \end{pmatrix}
                  \xrightarrow{\iota}
                  \begin{pmatrix}
                      2 & 0 \\
                      0 & 3
                  \end{pmatrix}
                  \quad\quad\quad
                  \begin{pmatrix}
                      i & 0 \\
                      0 & i
                  \end{pmatrix}
                  \xrightarrow{\iota}
                  \begin{pmatrix}
                      3 & 0 \\
                      0 & 3
                  \end{pmatrix};
              \end{align*}

        \item $G=C_{24}\rtimes C_{2}$
              \begin{align*}
                  \begin{pmatrix}
                      \zeta_6 & 0            \\
                      0       & \zeta_6^{-1}
                  \end{pmatrix}
                   & \xrightarrow{\iota}
                  \begin{pmatrix}
                      0 & 1 \\
                      4 & 1
                  \end{pmatrix}
                  \quad\quad\quad
                  \begin{pmatrix}
                      0       & -\zeta_6^{-1} \\
                      \zeta_6 & 0
                  \end{pmatrix}
                  \xrightarrow{\iota}
                  \begin{pmatrix}
                      0 & 2 \\
                      2 & 0
                  \end{pmatrix}
                  \quad\quad\quad
                  \begin{pmatrix}
                      \zeta_8 & 0        \\
                      0       & -\zeta_8
                  \end{pmatrix}
                  \xrightarrow{\iota}
                  \begin{pmatrix}
                      4 & 2 \\
                      3 & 1
                  \end{pmatrix}.
              \end{align*}

        \item $G=C_3\rtimes C_8$
              \begin{align*}
                  \begin{pmatrix}
                      \zeta_3 & 0            \\
                      0       & \zeta_3^{-1}
                  \end{pmatrix}
                   & \xrightarrow{\iota}
                  \begin{pmatrix}
                      1 & 4 \\
                      3 & 3
                  \end{pmatrix}
                  \quad\quad\quad
                  \begin{pmatrix}
                      0 & i \\
                      1 & 0
                  \end{pmatrix}
                  \xrightarrow{\iota}
                  \begin{pmatrix}
                      4 & 2 \\
                      3 & 1
                  \end{pmatrix}.
              \end{align*}
        \item $G=S_3$
              \begin{align*}
                  \begin{pmatrix}
                      \zeta_3 & 0            \\
                      0       & \zeta_3^{-1}
                  \end{pmatrix}
                   & \xrightarrow{\iota}
                  \begin{pmatrix}
                      2 & 1 \\
                      3 & 2
                  \end{pmatrix}
                  \quad\quad\quad
                  \begin{pmatrix}
                      0 & 1 \\
                      1 & 0
                  \end{pmatrix}
                  \xrightarrow{\iota}
                  \begin{pmatrix}
                      4 & 0 \\
                      0 & 1
                  \end{pmatrix}.
              \end{align*}

        \item $G=C_3\times C_8$
              \begin{align*}
                  \begin{pmatrix}
                      \zeta_3 & 0            \\
                      0       & \zeta_3^{-1}
                  \end{pmatrix}
                   & \xrightarrow{\iota}
                  \begin{pmatrix}
                      0 & 4 \\
                      1 & 4
                  \end{pmatrix}
                  \quad\quad\quad
                  \begin{pmatrix}
                      \zeta_8 & 0            \\
                      0       & \zeta_8^{-1}
                  \end{pmatrix}
                  \xrightarrow{\iota}
                  \begin{pmatrix}
                      4 & 2 \\
                      3 & 1
                  \end{pmatrix}.
              \end{align*}
        \item $G=C_3\times C_4$
              \begin{align*}
                  \begin{pmatrix}
                      \zeta_3 & 0            \\
                      0       & \zeta_3^{-1}
                  \end{pmatrix}
                   & \xrightarrow{\iota}
                  \begin{pmatrix}
                      1 & 2 \\
                      1 & 3
                  \end{pmatrix}
                  \quad\quad\quad
                  \begin{pmatrix}
                      i & 0 \\
                      0 & i
                  \end{pmatrix}
                  \xrightarrow{\iota}
                  \begin{pmatrix}
                      3 & 0 \\
                      0 & 3
                  \end{pmatrix}.
              \end{align*}

    \end{enumerate}
\end{lemma}

\begin{proof}
    Direct computation.
\end{proof}

\begin{lemma}\label{L: value of chi depending of mu4}
    Let $F/\ma Q_3$ be a finite extension with residual degree $f$, and $\chi:\cl O_F^{\times}\xrightarrow{}\ma C^{\times}$ a character of order $2d$ with $d$ odd. Then, $\chi(-1)=1$ if and only if $f\equiv0\pmod 2$; the latter condition is equivalent to $\mu_4\subset F$.
\end{lemma}

\begin{proof}
    Clearly $\chi(-1)=\pm 1$. If $f\equiv0\pmod 2$, then there is $\zeta_4\in \cl O_F^{\times}$ such that $\zeta_4^2=-1$ and so $\chi(-1)=\chi(\zeta_4^2)=\chi(\zeta_4)^2$. Since $\chi$ has order $2d$ and $\chi(\zeta_4)$ has order dividing $4$, we have $\chi(-1)=1$. Conversely, $f\equiv1\pmod 2$ implies that $\mu_4\not\subset F$, hence the group of units $\cl O_F^{\times}\simeq \mu_F \times \ma Z^{[F:\ma Q_3]}_3\simeq \langle-1\rangle\times \mu_r \times \ma Z^{[F:\ma Q_3]}_3$ where $r$ is odd. Therefore, the quotient of $\cl O_F^{\times}$ by any finite index subgroup $H$ containing $-1$ has odd order. So, if $-1\in\ker \chi$, then $ \cl O_F^{\times}/\ker \chi \simeq \Image\chi\simeq C_{2d}$ must have odd order, a contradiction. Thus $-1\notin\ker \chi$ and~$\chi(-1)=-1$.
\end{proof}

We now have all the tools to complete the proof of Theorem~\ref{T: Main Theorem
    for p=3}.

\begin{proof}[End of proof of Theorem~\ref{T: Main Theorem for p=3}]

    Let $f$ be the common residual degree of $F/\ma Q_3$ and $K/\Q_3$. We recall
    that $\chi_5(I_F)=1$ and $\chi_5(\Frob_F)=3^{-f}\in\F_5$ (note that $\Frob_F$
    is a geometric Frobenius).

    In what follows, we are going to define representations $\rhobar$ of $W_F$ with
    finite image. Thus, $\rhobar$ factors through a finite quotient $W_F/W_N\simeq
        G_F/G_N$. It follows that, if we extend $\rhobar$ to a representation of $W_F$
    or $G_F$, then it has open kernel. So $\rhobar$ is continuous and we can apply
    Proposition~\ref{Rubin}.

    Let us begin with case $(iii)$. Let $\tau$ be a inertial type as in the
    statement of this case and $\chi$ be an associated inertia character. We can
    extend $\chi$ to a character $\theta$ of $K^{\times}$ by defining $\theta(\pi)$
    for $\pi$ a uniformizer in $K$. For quadratic ramified extensions of $3$-adic
    fields, we can assume that $s(\pi)=-\pi$ so
    $\lambda:=\chi(\frac{\pi}{s(\pi)})=\chi(-1)$. Since the determinant of $\tau$
    is trivial and $\Image\tau=C_3\rtimes C_4$, then, using Lemma~\ref{L: Elliptitc
        Curve SCR conditions for chi}, $\chi$ satisfies the hypothesis of Lemma~\ref{L:
        Image of ramified inductions} with $d=6$. We can now use this to control the
    the image of $\rho$ by choosing $\theta(\pi)$ adequately. We now split the
    proof of $(iii)$ into 4 cases.

    If $f\equiv 1 \pmod 2$, then $\chi(-1)=-1$ by Lemma~\ref{L: value of chi
        depending of mu4}, and we define $\theta(\Frob_K)=\zeta_8$. This gives that the
    image of $\rho$ is isomorphic to $C_{24}\rtimes C_2$ by Lemma~\ref{L: Image of
        ramified inductions} $(d)$. Let $\rhobar:=\iota\circ\rho$ where $\iota$ is the
    isomorphism of Lemma~\ref{L: injections p=3} (b). This gives $\det
        \rhobar(\Frob_K)=3\in\F_5$ and $\det\rhobar(I_F)=1$ which is $\chi_5$ if
    $f\equiv3\pmod 4$. If $f\equiv 1\pmod 4$ the same reasoning applies with
    $\rhobar$ replaced by $g\circ\rhobar$ where $g$ is the automorphism of
    $\GL_2(\F_5)$ define as $g(A)=(A^{-1})^{T}$. The conclusion follows from
    Proposition~\ref{Rubin}.

    If $f\equiv2 \pmod 4$, then $\lambda=\chi(-1)=1$ and we define
    $\theta(\Frob_K)=i$ so that the image of $\rho$ is $C_4\times S_3$ by
    Lemma~\ref{L: Image of ramified inductions} (e). We consider
    $\rhobar=\iota\circ\rho$ where $\iota$ is the injection in Lemma~\ref{L:
        injections p=3} (a) thus $\det\rhobar(\Frob_K)=4$ and $\det\rhobar|_{I_F}=1$
    and we finish as before.

    Finally, if $f\equiv0\pmod 4$, then by choosing $\theta(\Frob_K)=1$ the image
    of $\rho$ is $C_3\rtimes C_4$. Since $\chi_5$ is trivial in this case, any
    injection $\iota : C_3\rtimes C_4 \hookrightarrow \SL_2(\ma F_5)$ works,
    finishing the proof of~$(iii)$.

    The cases $\Image\tau=C_2$ and $\Image \tau =C_4$ are treated in Lemma~\ref{Pot
        Good Uniform} and Lemma~\ref{L: tame for p=2 and p=3} respectively. It remains
    to prove the wild cases of $(i)$ and $(ii)$. By Lemma~\ref{L: Quadratic Twists
        of Ell Curves} all $C_6$ inertial types are quadratic twists of $C_3$ types so
    it is enough to deal with this case. The proof of cases $(i)$ and $(ii)$ is
    analogous to $(iii)$ but easier.

    For the proof of $(i)$ let $\tau$ be a principal series with unramified
    determinant and image $C_3$ and let $\chi$ be an associated inertia character.
    By fixing the value of $\lambda:=\theta(Frob_F)$ we can extended $\chi$ to a
    character $\theta$ of $W_F$. Let $\rho:=\theta\oplus\theta^{-1}$ we are going
    to define now the value of $\lambda$ and a map $\iota$ so that
    $\rhobar:=\iota\circ\rho$ has determinant $\chi_5$. If $f\equiv1\pmod 2$ set
    $\theta(\Frob_F)=\zeta_8$, $\iota$ be the injection in~\ref{L: injections p=3}
    (e) then, either $\det\rhobar=\chi_5$ or $\det(g\circ\rhobar)=\chi_5$. If
    $f\equiv2\pmod 4$ set $\theta(\Frob_F)=i$ and $\iota$ to be the injection
    in~\ref{L: injections p=3} (f). Finally if $f\equiv0\pmod 4$ set
    $\theta(\Frob_F)=1$ and $\iota$ be any injection of $C_3$ into $\GL_2(\F_5)$.
    In both cases $\det\rhobar=\chi_5$.

    For the proof of $(ii)$ let $\tau$ be a $C_3$ unramified supercuspidal type
    with trivial determinant and let $\chi$ be an associated inertia character. As
    before we can extend $\chi$ to $\theta$ by choosing the image of any Frobenius
    of $K$. Since $K/F$ is unramified we can choose $Frob_F^2$ as a Frobenius of
    $K$ and use Lemma~\ref{L: Image of unramified inductions} to control the image
    of $\rho$. Let $\rhobar:=\iota\circ\rho$, from here on the proof is completely
    analogous to the previous ones we just present the relevant choices. If
    $f\equiv1\pmod 2$ set $\theta(\Frob_F^2)=i$ and $\iota$ the injection in
    Lemma~\ref{L: injections p=3} $(c)$ then either $\rhobar$ or $g\circ\rhobar$
    has determinant $\chi_5$. If $f\equiv2\pmod 4$ set $\theta(\Frob_F^2)=1$ and
    $\iota$ as in Lemma~\ref{L: injections p=3} (d). Finally if $f\equiv0\pmod 4$
    set $\theta(\Frob_F^2)=-1$ and iota any injection of $C_3\rtimes C_4$ in
    $\SL_2(\F_5)$.

    In both the principal series and supercuspidal cases the result follows by
    Proposition~\ref{Rubin}.
\end{proof}

\subsection{Sufficient conditions for $p=2$} To finish the proof of Theorem~\ref{T: Main Theorem for p=2}, we will use
$\ell=3$ to build the representation $\bar{\rho}:=\iota\circ\rho$ as in
Proposition~\ref{Rubin}. The strategy is the same as for $p=3$, but more
involved due to the existence of exceptional inertial types. Observe that
Lemma~\ref{Pot Good Uniform} solves the case $C_2$ and Lemma~\ref{L: Quadratic
    Twists of Ell Curves} reduces the case $C_6$ to the case $C_3$ which is treated
in Lemma~\ref{L: tame for p=2 and p=3}. Thus it remains to prove the case $C_4$
of $(i)$ and $(ii)$ and cases $(iii)$ and $(iv)$. We will need the following
preparatory results.

\begin{lemma}\label{L: injections p=2}
    Let $G$ be a group.
    The maps $\iota$ defined on generators of $G$ is an isomorphism from $G$ as a subgroup of $\GL_2(\ma C)$ to $G$ as a subgroup of $\GL_2(\ma F_3)$.
    \begin{enumerate}[(a)]
        \item $G=\SD_{16}$
              \begin{align*}
                  \begin{pmatrix}
                      i & 0  \\
                      0 & -i
                  \end{pmatrix}
                   & \xrightarrow{\iota}
                  \begin{pmatrix}
                      0 & 2 \\
                      1 & 0
                  \end{pmatrix}
                  \quad\quad\quad
                  \begin{pmatrix}
                      0 & i \\
                      i & 0
                  \end{pmatrix}
                  \xrightarrow{\iota}
                  \begin{pmatrix}
                      1 & 2 \\
                      2 & 2
                  \end{pmatrix}
                  \quad\quad\quad
                  \begin{pmatrix}
                      \zeta_8 & 0         \\
                      0       & \zeta_8^3
                  \end{pmatrix}
                  \xrightarrow{\iota}
                  \begin{pmatrix}
                      2 & 2 \\
                      1 & 2
                  \end{pmatrix}
              \end{align*}
        \item $G=D_4$
              \begin{align*}
                  \begin{pmatrix}
                      i & 0  \\
                      0 & -i
                  \end{pmatrix}
                   & \xrightarrow{\iota}
                  \begin{pmatrix}
                      0 & 2 \\
                      1 & 0
                  \end{pmatrix}
                  \quad\quad\quad
                  \begin{pmatrix}
                      0 & 1 \\
                      1 & 0
                  \end{pmatrix}
                  \xrightarrow{\iota}
                  \begin{pmatrix}
                      1 & 0 \\
                      0 & 2
                  \end{pmatrix}
              \end{align*}
        \item $G=C_8$
              \begin{align*}
                  \begin{pmatrix}
                      \zeta_{8} & 0            \\
                      0         & \zeta_8^{-1}
                  \end{pmatrix}
                  \xrightarrow{\iota}
                  \begin{pmatrix}
                      1 & 2 \\
                      1 & 1
                  \end{pmatrix}
              \end{align*}
    \end{enumerate}
\end{lemma}
\begin{proof}
    Direct computation.
\end{proof}

\begin{lemma}\label{L: key lemma about exceptionals}
    Let $L/F$ satisfy {\bf (H1)} and $K_i$ satisfy {\bf (H2)} and let $M$ be the compositum of the $K_i$. Let $\rho:W_L\xrightarrow{}\GL_2(\ma C)$ be a triply imprimitive representation induced form $K_1$ by a character $\theta$. Assume $\rho(W_L)=Q_8$. Then, $\rho^{\sigma}\simeq\rho$ for all $\sigma\in\Gal(L/F)$ if and only if $\theta^{\sigma}\vert_{W_M}=\theta\vert_{W_M}$ for all $\sigma\in\Gal(M/F)$. In particular, $\rho$ is induced from $K_1$, $K_2$
    and~$K_3$. Moreover, if $\tilde{\theta}$ is another character such that $\tilde{\theta}\vert_{I_{K_1}}=\theta\vert_{I_{K_1}}$ then $\tilde{\theta}^{\sigma}\vert_{W_M}=\tilde{\theta}\vert_{W_M}$.
\end{lemma}
\begin{proof}
    Assume that $\rho^{\sigma}\simeq\rho$ for all $\sigma\in\Gal(L/F)$.
    Since the action of $\sigma\in\Gal(M/F)$ on $\rho$ depends only on the class of $\sigma$ in $\Gal(L/F)$, we also have  $\rho^{\sigma}\simeq\rho$ for all $\sigma\in\Gal(M/F)$. Note also that $\theta$ has order 4 because $\Image \rho \simeq Q_8$ and
    \begin{equation}\label{eq:induction}
        \rho^{\sigma}=(\Ind_{W_{K_i}}^{W_L} \theta)^{\sigma}=\Ind_{W_{\sigma(K_i)}}^{W_L} \theta^{\sigma}.
    \end{equation}
    From $\rho\simeq\rho^{\sigma}$ and {\bf (H2)}, it follows that $M$ is contained in the common field fixed by all $\rho^\sigma$. Since the field cut out by an induction is the field cut out by the induced character, we have that~$\theta\vert_{W_M}$ and $\theta^{\sigma}\vert_{W_M}$ cut out the same quadratic extension of $M$, equivalently,
    $\theta\vert_{W_M} = \theta^{\sigma}\vert_{W_M}$.

    For the converse, assume $\theta^{\sigma}\vert_{W_M}=\theta\vert_{W_M}$ for all
    $\sigma\in\Gal(M/F)$. Note that, if $L/F$, $N/L$ are Galois extensions and
    $\bar{\sigma}\in\Gal(\overline{F}/F)$, then the field $\bar{\sigma}(N)$ depends
    only on the residue class of $\bar{\sigma}$ in $\Gal(L/F)$. Therefore, in what
    follows, given $\sigma\in\Gal(L/F)$, we will write~$\sigma(N)$ for the field
    $\bar{\sigma}(N)$, where $\bar{\sigma}$ is any extension of $\sigma$
    to~$\overline{F}$.

    The representation $\rho^{\sigma}$ satisfies~\eqref{eq:induction} and has image
    $Q_8$ (as conjugation by $\sigma$ is an automorphism of $W_L \subset W_F$, the
    quotients of $W_L$ by $\ker\rho$ and by $\ker\rho^{\sigma}=\sigma \cdot
        \ker\rho \cdot \sigma^{-1}$ are isomorphic), therefore $\rho^{\sigma}$ is
    triply imprimitive induced from $\sigma(K_1)$ and two other fields. Let $N$ be
    the field cut out by $\rho$, and so the field cut out by $\rho^{\sigma}$ is
    $\sigma(N)$. We claim that $\sigma(N)=N$ for all $\sigma\in\Gal(M/F)$. Hence,
    since $\rho$ and $\rho^{\sigma}$ both have image~$Q_8$, we conclude from
    Lemma~\ref{L: Q8 has only one rep} that $\rho \simeq \rho^\sigma$. In
    particular, $\rho$ can be induced form any field of the form~$\sigma(K_1)$ with
    $\sigma\in \Gal(M/F)$ which by hypothesis {\bf (H2)} are $K_1,K_2$ and $K_3$.

    We now prove the claim which concludes the proof of the converse. Let $\sigma
        \in \Gal(M/F)$.

    If $M\subset N$, then $\sigma(M)=M\subset\sigma(N)$. Hence $N$ is the field cut
    out by $\theta\vert_{W_M}$ and $\sigma(N)$ is the field cut out by
    $\theta^{\sigma}\vert_{W_M}$. Since
    $\theta\vert_{W_M}=\theta^{\sigma}\vert_{W_M}$ by hypothesis, we have
    $\sigma(N) = N$.

    To complete the proof of the claim, we now assume $M\not\subset N$ and will
    obtain a contradiction. In this case, the field cut by $\theta\vert_{W_M}$ is
    $N\cdot M$ and the field cut out by $\theta^{\sigma}\vert_{W_M}$ is
    $\sigma(N)\cdot M$. From $\theta\vert_{W_M}=\theta^{\sigma}\vert_{W_M}$, we get
    $N\cdot M=\sigma(N)\cdot M$. Thus, for all $\sigma \in G_F$, we have
    $$\sigma(N\cdot M) = \sigma(N)\cdot \sigma(M)= \sigma(N) \cdot M = N\cdot M,$$
    so $N\cdot M$ is a Galois extension of~$F$, and also of~$L$; this is summarized
    in diagram~\ref{diagram}.
    \begin{figure}[h!]
        \[
            \xymatrix{
                & N\cdot M \ar@{-}[dl] \ar@{-}[dr] \ar@{-}[dd] & \\
                N \ar@{-}[dd] & & \sigma(N) \ar@{-}[dd] \\
                & M \ar@{-}[dd] \ar@{-}[dl] \ar@{-}[dr] & \\
                K_1 &  & \sigma(K_1) \\
                &L \ar@{-}[ul] \ar@{-}[ur]&
            }
        \]\caption{}\label{diagram}
    \end{figure}

    We have $\Gal(N/L)=Q_8$ and also $[NM:N]=2$ since $M\not\subset N$. It follows
    that $\Gal(NM/L)$ is an extension of
    \href{https://people.maths.bris.ac.uk/~matyd/GroupNames/1/e5/Q8byC2.html}{$Q_8$
        by $C_2$}. Thus $\Gal(NM/L)$ is isomorphic to $C_2\times Q_8$ or $C_4\rtimes
        C_4$. From ${\bf (H1)}$ and ${\bf (H2)}$, there exists $\sigma\in\Gal(L/F)$
    such that $\sigma(K_1)=K_2$, $\sigma(K_2)=K_3$ and $\sigma^3=1$. The fields $N
        \supset K_1$, $\sigma(N) \supset K_2$ and $\sigma^2(N) \supset K_3$ are
    different and, by the Galois correspondence, the group $\Gal(NM/L)$ contains
    three different normal subgroups of order 2 with $Q_8$ quotients. However, the
    groups \href{https://www.lmfdb.org/Groups/Abstract/16.12}{$C_2\times Q_8$} and
    \href{https://www.lmfdb.org/Groups/Abstract/16.4}{$C_4\rtimes C_4$} have, at
    most, two such normal subgroups, a contradiction.

    To prove the last statement, we check that
    $\theta\vert_{W_M}=\theta^{\sigma}\vert_{W_M}$ depends only on $\theta$
    restricted to~$I_{K_1}$. Indeed, the condition
    $\theta\vert_{W_M}=\theta^{\sigma}\vert_{W_M}$ is equivalent to the combination
    of $\theta\vert_{I_M}=\theta^{\sigma}\vert_{I_M}$ and
    $\theta(\Frob_M)=\theta^{\sigma}(\Frob_M)$ for any Frobenius element $\Frob_M
        \in W_M$. The first condition clearly follows from $\theta\vert_{I_{K_1}} =
        \theta\vert_{I_{K_1}}^\sigma$ and, for the second condition, we note that $$
        \theta(\Frob_M)=\theta^{\sigma}(\Frob_M)\iff
        \theta(\Frob_M)=\theta(\sigma^{-1}\Frob_M\sigma)\iff
        \theta(\Frob_M\sigma^{-1}\Frob_M^{-1}\sigma)=1 $$ and the conclusion follows as
    $\Frob_M\sigma^{-1}\Frob_M^{-1}\sigma\in I_{K_1}$; indeed, $W_F/I_F \simeq
        \hat{\ma Z}$ is abelian, thus the class of
    $\Frob_M\sigma^{-1}\Frob_M^{-1}\sigma$ is trivial in it and, moreover, since
    $M/F$ is Galois, it also belongs to $W_M$ thus it belongs to $W_M \cap I_F =
        I_M \subset I_{K_1}$.
\end{proof}

\begin{proposition}\label{P: Creation of exceptionals}
    Let $L,\rho$ and $\tau$ satisfy the hypothesis of Theorem~\ref{T: Main Theorem for p=2} $(iv)$. Then, there exists an exceptional representation $\tilde{\rho}$ of $F$ such that $\tilde{\rho}\vert_{I_F}=\tau$ and:\begin{itemize}
        \item If $\mu_3\subset F$ and $L/F$ is cubic unramified then $\tilde{\rho}(I_F)=Q_8$
              and $\tilde{\rho}(W_F)=\SL_2(\ma F_3)$
        \item If $\mu_3\subset F$ and $L/F$ is cubic ramified then
              $\tilde{\rho}(I_F)=\SL_2(\ma F_3)$ and $\tilde{\rho}(W_F)=\SL_2(\ma F_3)$
        \item If $\mu_3\not\subset F$ then $\tilde{\rho}(I_F)=\SL_2(\ma F_3)$ and
              $\tilde{\rho}(W_F)=\GL_2(\ma F_3)$ or $\tilde{\rho}(W_F)=2O$, the binary
              octahedral group.
    \end{itemize}
    Moreover, we have the following conductor exponent relation
    \begin{equation}\label{Eq: m for exceptionals}
        m(\tilde{\rho})=\frac{m(\rho)-2}{e(L/F)}+2.
    \end{equation}

\end{proposition}

\begin{proof}
    The assumption $\rho^{\sigma}\simeq\rho$ for all~$\sigma\in\Gal(L/F)$ ensures that $\rho$ extends to a representation $\tilde{\rho}:W_F\xrightarrow{}\GL_2(\ma C)$ (see~\cite[42.4 Proposition]{Bushnell2006}). Let $\tilde{N}$ be the field fixed by~$\ker \tilde{\rho}$.

    Suppose first that $\mu_3\subset F$. Let $N/L$ extension fixed by the kernel of
    $\rho$. We claim that $L\subset\tilde{N}$. Otherwise, $L\cap \tilde{N}=F$ hence
    $\Image \tilde{\rho}\simeq\Image\rho\simeq Q_8$ and $\Gal(N/F)=C_3\times Q_8$.
    A quick look at the group
    \href{https://www.lmfdb.org/Groups/Abstract/24.11}{$C_3\times Q_8$} in
    LMFDB~\cite{lmfdb} shows that all subgroups of this group are normal; thus, all
    intermediate fields $F\subset K\subset N$ are Galois over $F$, contradicting
    the assumption {\bf (H2)}, proving the claim. Thus $N=\tilde{N}$.

    In this situation, $\Gal(N/F)\simeq\Image\tilde{\rho}$ is an extension of
    \href{https://people.maths.bris.ac.uk/~matyd/GroupNames/1/C3byQ8.html}{$C_3$ by
        $Q_8$}, so it is either $Q_8\times C_3$ or~$\SL_2(\F_3)$. From the above, the
    former is not possible, thus $\Gal(N/F)\simeq \SL_2(\F_3)$. Since $L\subset
        \tilde{N}$, it is clear that, if $L/F$ is unramified, then
    $\tilde{\rho}(I_F)=Q_8$ and $\tilde{\rho}\vert_{I_F}=\tilde{\rho}\vert_{I_L}=
        \rho|_{I_L}\simeq\tau$. If $L/F$ is ramified, then $\tilde{\rho}(I_F)=\SL_2(\ma
        F_3)$ and, since $\det\tilde{\rho}\vert_{I_L}=\det\tau\vert_{I_L}=1$, either
    $\tilde{\rho}$ or $\psi_L^{\pm1}\tenso\tilde{\rho}$ has unramified determinant.
    Thus, by Lemma~\ref{L: if SL23 coincide on Q8 are equal}, one of
    $\tilde{\rho},\psi_L\tenso\tilde{\rho}$ or $\psi_L^{-1}\tenso\tilde{\rho}$ has
    inertia type $\tau$. These three representations have image $\SL_2(\F_3)$ and
    the same projective image, so the conclusion follows.

    Let us assume now that $\mu_3\not\subset F$. Set
    $\rho_0:=\tilde{\rho}\vert_{W_{F(\mu_3)}}$ and let $N$ be the field fixed by
    the kernel of $\rho_0$. Since $L/F(\mu_3)$ is a cubic ramified extension, by
    the previous case, we have that $\Image\rho_0\simeq \SL_2(\ma F_3)\simeq
        \Gal(N/F(\mu_3))$ and $L\subset N$. We claim $F(\mu_3)\subset\tilde{N}$.
    Otherwise $F(\mu_3)\cap \tilde{N}=F$ hence $\Gal(N/F)\simeq \SL_2(\ma
        F_3)\times C_2$. But this group does not have an $S_3\simeq\Gal(L/F)$ quotient,
    proving the claim. Then $\tilde{N}=N$ and $\Gal(N/F)$ is an extension of
    \href{https://people.maths.bris.ac.uk/~matyd/GroupNames/1/C2bySL(2,3).html}
    {$C_2$ by $\SL_2(\ma F_3)$} that has an $S_3$ quotient; this implies
    $\Gal(N/F)\simeq \GL_2(\ma F_3)$ or $\Gal(N/F)\simeq 2O$. Note that, when
        $\mu_3\not\subset F$, condition {\bf (H1)} for $L/F$ implies $I_F/I_L\simeq
            C_3$. In particular, we can apply a combination of Lemma~\ref{L: if SL23
            coincide on Q8 are equal} and Lemma~\ref{L:Exceptionals have unramified det} to
        deduce that $\tilde{\rho}\vert_{I_F}\simeq \tau$.

    For $G\in\{\SL_2(\F_3),\GL_2(\F_3),2O\}$, we have $G/\Center(G)\in\{A_4,S_4\}$
    so this representations are exceptional. The equation for the conductors
    follows from a simple computation using the definition of conductor (see
    \cite[Chapter IV,\S 10]{Silverman1994}), and the fact that $L/F$ is at worst
    tamely ramified.
\end{proof}

Recall that for a representation $\rho : W_F \to \GL_2(\ma C)$ and a finite
extension $L/F$ we write $\rho_L$ for the restriction of $\rho$ to~$W_L$.

\begin{proposition}\label{P: exceptional image switch}
    Let $F/\Q_2$ be a finite extension such that $\mu_3\not\subset F$. Let $\rho:W_F\xrightarrow{}\GL_2(\ma C)$ be an exceptional representation satisfying $\rho(I_F) \simeq \SL_2(\F_3)$ and $\rho(W_F)\simeq 2O$ or $\GL_2(\F_3)$. Then, there exists an exceptional representation $\tilde{\rho}$ of $W_F$ such that $\tilde{\rho}\vert_{I_F} = \rho\vert_{I_F}$ and $\tilde{\rho}(W_F)\simeq\GL_2(\F_3)$ or $2O$, respectively. \end{proposition}

\begin{proof}
    We will prove the case $\rho(W_F)\simeq2O$. Let $N:=\Fbar^{\ker \rho}$.
    By the Galois correspondence, we have the following diagrams of subfields of $N$ and subgroups of~\href{https://www.lmfdb.org/Groups/Abstract/diagram/48.28}{$2O$}, where $L/F$ satisfies {\bf (H1)} and $K_i$ satisfy {\bf (H2)}.

    \begin{figure}[h]
        \[
            \xymatrix{
            &&N\ar@{-}[d]&&&&&&1\ar@{-}[d]&\\
            &&M\ar@{-}[dr]\ar@{-}[d]\ar@{-}[dl]&&&&&&C_2\ar@{-}[dr]\ar@{-}[d]\ar@{-}[dl]&\\
            &K_1\ar@{-}[dr]\ar@{-}[d]\ar@{-}[dl]&K_2\ar@{-}[d]&K_3\ar@{-}[dl]&&&
            &C_4\ar@{-}[dr]\ar@{-}[d]\ar@{-}[dl]&C_4\ar@{-}[d]&C_4\ar@{-}[dl]\\
            L_3\ar@{-}[dr]&L_2\ar@{-}[d]&L\ar@{-}[dr]\ar@{-}[dl]&&&&
            Q_8\ar@{-}[dr]&C_8\ar@{-}[d]&Q_8\ar@{-}[dr]\ar@{-}[dl]&&\\
            &C\ar@{-}[dr]&&F(\mu_3)\ar@{-}[dl]&&&
            &Q_{16}\ar@{-}[dr]&&\SL_2(\F_3)\ar@{-}[dl]&\\
            &&F&&&&&&2O&
            }
            \]
    \end{figure}
    Since $\rho$ is exceptional, from Proposition~\ref{P: base change of exceptional types}, we know that $\rho$ becomes triply imprimitive over~$L$ and simply
    imprimitive over $C$. Moreover, from the lattice of subgroups, we see that~$\rho_L$ must be induced
    from the three~$K_i$ and $\rho_C$ is induced from~$L_2$ by a character $\theta$ of order 8. Since
    $\operatorname{Ind}_{W_{K_1}}^{W_L}\theta\vert_{W_{K_1}}$
    cuts out the field $N$, by Lemma~\ref{L: Q8 has only one rep}, we have
    $\rho_L \simeq \operatorname{Ind}_{W_{K_1}}^{W_L}\theta\vert_{W_{K_1}}$.

    We set $\chi:=\theta\vert_{I_{K_1}}$.

    The extension $K_1/L$ is ramified while $K_1/L_i$ is unramified for $i=2,3$, so
    $I_{K_1} = I_{L_2} = I_{L_3}$. Let $s_i:=\Frob_{L_i} \in G_{L_i} \subset G_C$
    be Frobenius elements of $L_i$ and so also of~$C$; in particular, they are
    lifts in $\Gal(\bar{C}/C)$ of the generators of $\Gal(K_1/L_2)$ and
    $\Gal(K_1/L_3)$, respectively. Then, $s_{L}:=s_2\cdot s_3$ is a lift of the
    generator of $\Gal(K_1/L)$. From Lemma~\ref{L:Exceptionals have unramified
        det}, we know that $\det \rho$ is unramified, hence $\det \rho_L = 1$ because
    $N/L$ is totally ramified. Thus $\chi^{s_L}=\chi^{-1}$ by
    Lemma~\ref{cyclotomicdeterminant}. Since $\theta$ is a character of~$W_{L_2}$,
    it follows that $\theta^{s_2}=\theta$. Thus, $\chi^{s_3}=\chi^{-1}$ and so we
    can apply Lemma~\ref{L: Image of unramified inductions} to the induction
    $\operatorname{Ind}_{W_{K_1}}^{W_{L_3}}\theta\vert_{W_{K_1}}$, to deduce that
    $\theta\vert_{W_{K_1}}(s_3^2)=-1$.

    Note that $s_3^2$ is a Frobenius element in $W_{K_1}$. Consider the character
    $\tilde{\theta}$ of $W_{K_1}$ defined by $\tilde{\theta}\vert_{I_{K_1}}:=\chi$
    and $\tilde{\theta}(s_3^2)=1$; denote by $\tilde{N}$ the filed fixed by $\ker
        \tilde{\theta}$. From above, we have $(\tilde{\theta}|_{I_{K_1}})^{s_3} =
        \chi^{s_3}=\chi^{-1}$, and so, by Lemma~\ref{L: Image of unramified
        inductions}, the image of
    $\operatorname{Ind}_{W_{K_1}}^{W_{L_3}}\tilde{\theta}$ is $D_4 \simeq
        \Gal(\tilde{N}/L_3)$.

    We now focus on the representation
    $\tilde{\rho}_L:=\operatorname{Ind}_{W_{K_1}}^{W_L}\tilde{\theta}$. We have
    $\tilde{\rho}_L |_{I_L} \simeq \operatorname{Ind}_{I_{K_1}}^{I_L}\chi \simeq
        \rho_L|_{I_L} \simeq Q_8$ and, since $\tilde{\rho}$ is the induction the order
    4 character $\tilde{\theta}$, we also have $\Image \tilde{\rho}_L \simeq Q_8$.
    Now, since $\rho_L^{\sigma}=\rho_L$ for all $\sigma \in \Gal(L/F)$, by
    Lemma~\ref{L: key lemma about exceptionals}, we have
    $\theta^{\sigma}\vert_{W_M}=\theta\vert_{W_M}$. By the same lemma, since
    $\tilde{\theta}\vert_{I_{K_1}}=\theta\vert_{I_{K_1}}$ we also have
    $\tilde{\theta}^{\sigma}\vert_{W_M}=\tilde{\theta}\vert_{W_M}$ which again by
    Lemma~\ref{L: key lemma about exceptionals} implies
    $\tilde{\rho}_L^{\sigma}\simeq\tilde{\rho}_L$ for all $\sigma \in \Gal(L/F)$.
    From Proposition~\ref{P: Creation of exceptionals}, it follows that
    $\tilde{\rho}_L$ descends to a representation~$\tilde{\rho}$ of~$W_F$ with
    image either $2O$ or $\GL_2(\F_3)$. By construction, the field $\tilde{N}$ cut
    out by $\tilde{\rho}$ contains $L_3$ and thus $\Image\tilde{\rho}$ contains a
    $\Gal(\tilde{N}/L_3)\simeq D_4$ subgroup. This is only possible if $\Image
        \tilde{\rho}\simeq
        \href{https://www.lmfdb.org/Groups/Abstract/diagram/48.29}{\GL_2(\F_3)}$.
    Finally, observe that $\tilde{\rho}$ contains the matrix~$\left( \begin{smallmatrix} -1 & 0 \\ 0 & -1 \end{smallmatrix}\right)$ in the image of inertia, hence $\PP\tilde{\rho} \simeq 2O/\Center(2O) \simeq S_4$ so $\tilde{\rho}$ is an exceptional representation. Since $\rho\vert_{I_L}\simeq\tilde{\rho}\vert_{I_L}$ by Lemma~\ref{L: if SL23
        coincide on Q8 are equal} and Lemma~\ref{L:Exceptionals have unramified det} it follows that $\rho\vert_{I_F}\simeq\tilde{\rho}\vert_{I_F}$, completing the proof.

    The case $\rho(W_F)\simeq\GL_2(\F_3)$ follows analogously by noting that the
    only difference in the subgroup diagram is that the leftmost group $Q_8$ is
    replaced by a $D_4$.
\end{proof}

From the previous two propositions we obtain the following corollary.

\begin{corollary}\label{C: We can choose GL23}
    Let $L,\rho$ and $\tau$ satisfy the hypothesis of Theorem~\ref{T: Main Theorem for p=2} $(iv)$. Assume further that $\mu_3\not\subset F$. Then, there exists an exceptional representation $\tilde{\rho}$ of $W_F$ such that $\tilde{\rho}\vert_{I_F}=\tau$, $\tilde{\rho}(I_F)=\SL_2(\ma F_3)$ and $\tilde{\rho}(W_F)=\GL_2(\ma F_3)$.
\end{corollary}

\begin{example}
    The previous proposition predicts the existence of a Galois extensions $N/\Q_2$ with $\Gal(N/\Q_2)=2O$ over which the exceptional elliptic curves appearing in~\cite[Table 17]{dembélé2024galoisinertialtypeselliptic}
    attain good reduction. For example, the curve
    \href{https://www.lmfdb.org/EllipticCurve/Q/648/a/1}{648.a1} with model $E:y^2 =
        x^3-3x-1$ obtains good reduction over the
    extension $N/\Q_2$ defined by the polynomial
    \[x^{48} + 4x^{46} + 6x^{44} + x^{40} + 2x^{36} + 4x^{34} + 2x^{32} +
        4x^{30} + 6x^{28} + 2x^{25} + x^{24} + 4x^{23} + 2x^{21} \]\[+ 6x^{20} +
        4x^{19} + 6x^{18} + 6x^{17} + 4x^{16} + 4x^{15} + 6x^{13} + 6x^{9} +
        3x^{8} + 6x^{6} + 2x^{5} + 6x^{4} + 3.\] This field has Galois group $2O$ and ramification index $24$. For the list of
    analogous fields for the other exceptional curves see~\cite[\href{https://github.com/3nr1c/inertial-types/blob/main/data/2OFields.m}{List of $2O$ fields}]{MAGMAFiles}
\end{example}

We can finish the proof of Theorem~\ref{T: Main Theorem for p=2}.

\begin{proof}[End of proof of Theorem~\ref{T: Main Theorem for p=2}]
    The proof of the $C_4$ case in parts $(i)$ and $(ii)$ is analogous to the corresponding cases for $p=3$, where we use the following choices instead:

    Case $f\equiv 1\pmod2$: for $\tau$ a principal series, take
    $\theta(Frob_F)=\zeta_8$, $\rho=\theta\oplus\theta^{-1}$ and $\iota$ the
    injection in Lemma~\ref{L: injections p=2}~(c); for $\tau$ unramified
    supercuspidal, take $\theta(\Frob_F^2)=1$, $\rho=\Ind_{W_K}^{W_F}\theta$ and
    $\iota$ the injection in Lemma~\ref{L: injections p=2} (b).

    Case $f\equiv 0\pmod 2$: for $\tau$ a principal series, take
    $\theta(\Frob_F)=1$, $\rho=\theta\oplus\theta^{-1}$ and $\iota$ any injection
    of $C_4$ into $\GL_2(\ma F_3)$; for $\tau$ unramified supercuspidal, take
    $\theta(\Frob_F^2)=-1$, $\rho=\Ind_{W_K}^{W_F}\theta$ and $\iota$ any injection
    of $Q_8$ into $\GL_2(\ma F_3)$.

    Let us now prove case $(iii)$. Let $\tau$ satisfy its hypotheses and $\chi$ be
    an inertia character associated with~$\tau$. We can extend $\chi^A$ to a
    character $\theta^A$ of $K^{\times}$ by defining $\theta^A(\pi)$ for $\pi$ a
    uniformizer in $K$. By Lemma~\ref{L: explicit tiply conditions}, we know the
    value of $\lambda=\chi(\pi/s(\pi))$ in terms of $f$. Since the determinant of
    $\tau$ is trivial and $\Image \tau =Q_8$, using Lemma~\ref{L: Elliptitc Curve
        SCR conditions for chi}, we see that $\chi$ satisfies the hypothesis of
    Lemma~\ref{L: Image of ramified inductions} with $d=4$. We can now use this to
    control the the image of $\rho$ by choosing $\theta(\pi)$ adequately. We now
    split into 2 cases.

    If $f\equiv 1 \pmod 2$, then $\chi(\pi/s(\pi))=\pm i$. After replacing~$\pi$
    by~$u\pi$ where $u$ is a unit with $\chi(u)=i$, we can assume that
    $\chi(\pi/s(\pi))=-i$. We define $\theta(\Frob_K)=\zeta_8$. Thus $\Image \rho
        \simeq \SD_{16}$ by Lemma~\ref{L: Image of ramified inductions} $(c)$. Let
    $\rhobar:=\iota\circ\rho$ where $\iota$ is given in Lemma~\ref{L: injections
        p=2} (a). Hence $\det\rhobar(\Frob_K)=2\in\F_3$ and $\det\rhobar(I_F)=1$ which
    is $\chi_3$. The conclusion follows from Proposition~\ref{Rubin}.

    If $f\equiv0 \pmod 2$, then $\lambda=\pm 1$. We define $\theta(\Frob_K)=1$ if
    $\lambda=1$ or $\theta(\Frob_K)=i$ if $\lambda=-1$. Thus $\Image \rho \simeq
        Q_8$ by Lemma~\ref{L: Image of ramified inductions}~(a) and~(b), respectively.
    We consider $\rhobar=\iota\circ\rho$ where $\iota$ is any injection $\iota$ of
    $Q_8$ into $\SL_2(\ma F_3)$. Thus $\det\rhobar=\chi_3$, completing the proof of
    case~$(iii)$.

    Finally, we prove $(iv)$. Let $\tau$ satisfying its hypotheses; again, we split
    into two cases.

    If $f\equiv 0 \pmod 2$ (or equivalently $\mu_3\subset F$), then, by
    Proposition~\ref{P: Creation of exceptionals}, there is $\rho$ with
    $\rho\vert_{I_F}=\tau$ and $\Image \rho \simeq \SL_2(\ma F_3)$. Thus, any
    injection of $\Image \rho$ into $\GL_2(\ma F_3)$ automatically gives $\rhobar$
    with determinant $\chi_3$ which is trivial in this case.

    If $f\equiv 1\pmod2$, then, by Corollary~\ref{C: We can choose GL23}, there is
    $\rho$ with $\rho\vert_{I_F}=\tau$, $\rho(I_F)\simeq\SL_2(\ma F_3)$ and $\Image
        \rho \simeq \GL_2(\ma F_3)$. Let $\rhobar=\iota\circ\rho$ where $\iota$ is any
    isomorphism between $\Image \rho$ and $\GL_2(\ma F_3)$. Since the image of
    $\rhobar$ is generated by $\rhobar(I_F)$ and $\rhobar(\Frob_F)$ it follows that
    $\det\rhobar(\Frob_F)=-1$, otherwise the image of $\rhobar$ would be $\SL_2(\ma
        F_3)$.
\end{proof}

\begin{remark}
    Observe that the representations $\rho$ used in the proof of Theorems~\ref{T: Main Theorem for p=3} and~\ref{T: Main Theorem for p=2} are not Weil-Deligne representations of an elliptic curve. Indeed, $\rho$ has finite image while the Weil-Deligne representation associated to elliptic curves do not.
\end{remark}

\section{The main theorems in terms of characters}
\label{S: charactersformulation}

Our main theorems provide a set of necessary and sufficient conditions for an
inertial type to arise from an elliptic curve. To explicitly find such types,
we will express these conditions in terms of characters so that they can be
computed in a reasonable amount of time.

Recall form \S\ref{Notation} that $\varepsilon_K$ is the restriction to inertia
of the character $\psi_K$ of $W_F$ associated with a quadratic extension $K/F$.
Recall also Definition~\ref{D: Associated inertia characters} of inertia
character associated with an inertial type.

\begin{theorem}\label{T: explicit conditions p=3}
    Let $F/\Q_3$ be a finite extension.
    Let $K / F$ and $\chi^A$ satisfy one of the following:
    \begin{enumerate}[(i)]
        \item $\chi^A$ is a character of $\cl O_F^{\times}$ such that
              $\operatorname{Ord}(\chi^A)\in\{2,3,4,6\}$;
        \item  $K/F$ is quadratic unramified and $\chi^A$ is a character of
              $\cl O_K^{\times}$ such that
              $\operatorname{Ord}(\chi^A)\in\{3,4,6\}$ and
              $\chi^A\vert_{\cl O_F^{\times}}=\varepsilon_K^{A}$;
        \item $K/F$ is quadratic ramified and $\chi^A$
              is a character of $\cl O_K^{\times}$ such that
              $\operatorname{Ord}(\chi^A)=6$
              and $\chi^A\vert_{\cl O_F^{\times}}=\varepsilon_K^{A}$.
    \end{enumerate}
    Then, the inertia type associated to $\chi$ satisfies the case of Theorem~\ref{T: Main Theorem for p=3} with the same item number. Moreover, any $\tau$ satisfying the cases of Theorem~\ref{T: Main Theorem for p=3} has an associated inertia character satisfying the corresponding case above.
\end{theorem}

\begin{proof}
    Case $(i)$ is clear. For cases $(ii)$ and $(iii)$, note that the assumptions on the order of~$\chi^A$ implies that the induction $\Ind_{W_K}^{W_F} \theta$ of any character $\theta$ of $W_K$ restricting to $\chi$ on inertia is irreducible by Corollary~\ref{cor: the inductions are really irreducible}; moreover, from Lemma~\ref{cyclotomicdeterminant} we get the equivalence
    between~$\tau$ having trivial determinant and the condition $\chi^A\vert_{\cl O_F^{\times}}=\varepsilon_K^{A}$. Finally, for case (iii), from Lemma~\ref{L: Inductions have image Q8} applied with $H=I_K$ and $G=I_F$ we obtain the equivalence between the images of $\tau$ being $C_3\rtimes C_4$ and the order of $\chi^A$ being $6$.
\end{proof}

In the next statement, $\pi$, $\pi_{K_i}$ and $\pi_M$ denote a uniformizers in
$K$, $K_i$ and $M$, respectively. Moreover, $s \in W_F$ and $s_i \in W_L$ are
non-trivial lifts of the non-trivial elements in $\Gal(K/F)$ and $\Gal(K_i/L)$,
respectively.

\begin{theorem}\label{T: explicit conditions p=2}
    Let $F/\Q_2$ be a finite extension.
    Let $K / F$, $L/F$, $K_i / L$ and $\chi^A$ satisfy one of the following cases:
    \begin{enumerate}[(i)]
        \item $\chi^A$ is a character of $\cl O_F^{\times}$ such that
              $\operatorname{Ord}(\chi^A)\in\{2,3,4,6\}$.
        \item $K/F$ is quadratic unramified and $\chi^A$ is a character of $\cl O_K^{\times}$ such that $\operatorname{Ord}(\chi^A)\in\{3,4,6\}$ and
              $\chi^A\vert_{\cl O_F^{\times}}=\varepsilon_K^{A}$.
        \item $K/F$ is quadratic ramified and $\chi^A$
              is a character of $\cl O_K^{\times}$ such that $\operatorname{Ord}(\chi^A)=4$ and $\chi^A\vert_{\cl O_F^{\times}}=\varepsilon_K^{A}$. Assume further,
              \subitem (a) $\chi^A(\pi/s(\pi))=\pm 1$ if
              $\mu_3\subset F$ or
              \subitem (b) $\chi^A(\pi/s(\pi))=\pm i$
              if $\mu_3\not\subset F$.
        \item Let $L/F$ satisfy hypothesis {\bf (H1)} and let $K_1,K_2,K_3$ be quadratic
              extensions of $L$ satisfying {\bf (H2)}. Let $M$ be the compositum of the $K_i$
              and $U_M^i=\operatorname{Nm}^M_{K_i}(\cl O_M^{\times})$. Let $\chi^A$ be a
              character of any of the fields $K_i$ satisfying two conditions \subitem (a)
              $\operatorname{Ord}(\chi^A)=4$, $\chi^A\vert_{\cl
                  O_L^{\times}}=\varepsilon_{K_i}^{A}$ and $\chi^A(\pi_{K_i}/s_i(\pi_{K_i}))=\pm
                  1$; \subitem (b) for all $\sigma\in\Gal(M/F)$,
              $$\quad(\chi^A)\circ\operatorname{Nm}^{M}_{K_1}=(\chi^{\sigma})^A\circ\operatorname{Nm}^{M}_{\sigma(K_1)}\quad\text{and}\quad
                  \chi^A\circ \operatorname{Nm}(\pi_M/\sigma(\pi_M))=1.$$
    \end{enumerate}
    Then, the inertia type associated to $\chi$ satisfies the case of Theorem~\ref{T: Main Theorem for p=2} with the same item number. Moreover, any $\tau$ satisfying the cases of Theorem~\ref{T: Main Theorem for p=2} has an associated inertia character satisfying the corresponding case above.
\end{theorem}

\begin{proof}
    The proof of $(i),(ii)$ and the first statement of $(iii)$ is completely analogous to the same cases in Theorem~\ref{T: explicit conditions p=3}.
    For the subcases $(a)$ and $(b)$ of Theorem~\ref{T: Main Theorem for p=2} $(iii)$,
    Lemma~\ref{L: explicit tiply conditions} shows the equivalence between being simply/triply imprimitive and the value of $\chi^A(\pi/s(\pi))$.

    We will now prove $(iv)$. Let $\chi^A$ be as in its statement.

    Since $\mu_3 \subset L$, from assumption~(a), we see that the hypothesis of
    case (iii) including its subcase (a) hold for $K_i / L$ and $\chi^A$, which we
    proved to be equivalent to the assumptions of Theorem~\ref{T: Main Theorem for
        p=2}~(iii)~(a). Therefore, the ramified supercuspidal inertial type $\tau :=
        \Ind_{I_{K_i}}^{I_L} \chi$ is triply imprimitive with trivial determinant and
    image $Q_8$. Now, by Corollary~\ref{C: Rho with Q8 exits}, there is a triply
    imprimitive representation $\rho$ of $W_L$ with inertial type $\tau$ and
    $\Image \rho=Q_8$. Moreover, from the properties of the Artin map, it follows
    that assumption (b) is equivalent to
    $\theta\vert_{W_M}=\theta^{\sigma}\vert_{W_M}$. It now follows from
    Lemma~\ref{L: key lemma about exceptionals} that $\rho^{\sigma}\simeq\rho$ and
    $\rho$ is induced from $K_1,K_2$ and $K_3$.

    To prove the last statement, assume that $L/F$, $\rho$ and $\tau$ satisfies the
    hypothesis in Theorem~\ref{T: Main Theorem for p=2}~$(iv)$. In particular, we
    have $\rho|_{I_L} = \tau|_{I_L} \simeq \Ind_{I_{K_i}}^{I_L} \chi_i$ for $i=1,2$
    or $3$, where $\chi_i = \theta_i |_{I_{K_i}}$, and so we can take $\chi^A =
        \chi_i^A$ for any~$i$. It follows from case~$(iii)$ that assumption (a) holds.
    Finally, since $\rho^{\sigma}\simeq\rho$ for all $\sigma\in\Gal(L/F)$, again
    from Lemma~\ref{L: key lemma about exceptionals} we conclude that, for all
    $\sigma\in \Gal(M/F)$, we have $\theta_i\vert_{W_M} =
        \theta_i^{\sigma}\vert_{W_M}$, equivalently assumption (b) holds.

\end{proof}

\section{The algorithm and Examples}
\label{S: algo}

We have seen in Section~\ref{S: charactersformulation} that finding all
inertial types arising from elliptic curves defined over a 3-adic or 2-adic
field $F$ is equivalent to finding all characters satisfying the conditions in
Theorem~\ref{T: explicit conditions p=3} or Theorem~\ref{T: explicit conditions
    p=2}, respectively. In this section, we discuss our implementation of this idea
into an algorithm in {\tt Magma} together with examples of its application.

\begin{alg}\label{A: Algortihm for p=3} Let $\tau_E$ be the inertial type of an elliptic curve $E/F$.
    \begin{enumerate}
        \item For each class of inertial types, i.e, principal series, unramified
              supercuspidal, ramified supercuspidal and exceptional, we find all extensions
              of $F$ where an inertia character of the given class can exist.
        \item Using the bound $m(\tau_E)\leq2+3v_F(3)+5v_F(2)$ together with
              equations~\eqref{m for PS}--\eqref{m for SCR} and~\eqref{Eq: m for
                  exceptionals}, we obtain, for each extension $K/F$ computed in (1), an upper
              bound $m_K$ for the conductor of all possible inertia characters $\chi$ of $\cl
                  O_K^{\times}$ associated with~$\tau_E$.
        \item For each $K$ and corresponding~$m_K$, compute all characters of $(\cl
                  O_K^{\times}/\p^{m_K})^{\times}$ satisfying the hypothesis in Theorems~\ref{T:
                  explicit conditions p=3} and~\ref{T: explicit conditions p=2}.
        \item Finally, compute the chain of groups and maps $(\cl
                  O_K/\p^{(f+1)})^{\times}\xlongrightarrow{\varphi_f}(\cl O_K/\p^f)^{\times}$
              and, for each character computed in (3), we find its conductor, that is, we
              compute the smallest $f$ for which $\chi$ factors through $(\cl
                  O_K^{\times}/\p^f)$ via the map~$\varphi_f$.
    \end{enumerate}
\end{alg}

The following lemma shows that the condition $\chi^{A}\vert_{\cl
    O_F^{\times}}=\varepsilon_K^{A}$ ensuring (by
Lemma~\ref{cyclotomicdeterminant}) that the determinant of nonexceptional
supercuspidal types is trivial can be tested algorithmically.

\begin{lemma}\label{L: Is enough to check the ui}
    Let $\tau$ be a nonexceptional supercuspidal type defined over $F$ and~$\chi : I_K \to \ma C^\times$ its associated character with conductor exponent~$m(\chi)=m$. Then, the condition $\chi^{A}\vert_{\cl O_F^{\times}}=\varepsilon_K^{A}$ is equivalent to $\chi^A(u_i)=\varepsilon_K^{A}(u_i)$ for $u_i$ elements in $\cl O_F^{\times}$ whose image in $(\cl O_F/\ger{p}^r)^{\times}$ form a set of generators, for any $ r \geq \operatorname{max}\{\lceil \frac{m}{2} \rceil,m(\varepsilon_K)\}$.
\end{lemma}

\begin{proof}
    Since $r\geq m(\varepsilon_K)$, the character~$\varepsilon_K^A$ factors through the projection $\cl O_F^{\times}\xrightarrow{h_F}(\cl O_F/\ger{p}^r)^{\times}$, so $\varepsilon_K^A=\tilde{\varepsilon}_K\circ h_F$ for some character $\tilde{\varepsilon}_K$ of $(\cl O_F/\ger{p}^r)^{\times}$;
    similarly, $\chi=\tilde{\chi}\circ h_K$ where $\cl O_K^{\times}\xrightarrow{h_K}(\cl O_K/\ger{q}^m)^{\times}$. Now, since $r\geq\lceil\frac{m}{2}\rceil$, the inclusion $\cl O_F^{\times}\xrightarrow{i}\cl O_K^{\times}$ descends to the quotients, that is, there is a map $j : (\cl O_F/\ger{p}^r)^{\times} \hookrightarrow (\cl O_K/\ger{q}^m)^{\times}$ such that $h_K\circ i=j\circ h_F$. Then $\chi\vert_{\cl O_F^{\times}}=\chi\circ i=\tilde{\chi}\circ h_K\circ i=\tilde{\chi}\circ j\circ h_F$ and $\varepsilon_K^A=\tilde{\varepsilon}_K\circ h_F$ thus it is enough to check the condition on the generators of $(\cl O_F/\ger{p}^r)^{\times}$. Conversely, if $\chi^{A}\vert_{\cl O_F^{\times}}=\varepsilon_K^{A}$ then $\chi^A(u)=\varepsilon_K^{A}(u)$ for any $u\in\cl O_F^{\times}$, in particular, for a set of generators.
\end{proof}

\begin{theorem} Let $p=2$ or $p=3$. Let $F/\Q_p$ be a finite extension.
    The procedure in Algorithm~\ref{A: Algortihm for p=3} computes all inertial types arising from elliptic curves defined over $F$ with potentially good reduction together with their conductor. It always finishes.
\end{theorem}
\begin{proof}
    For each class of inertial type  there is a finite set of fields $K$ for admitting inertia characters; indeed, for principal series we have $K=F$, for unramified supercuspidal $K$ is the unique unramified quadratic extension of $F$, for ramified supercuspidal $K/F$ is ramified quadratic and for the exceptional types, the field $K$ is one of $K_1, K_2, K_3$ satisfying property~{\bf (H2)}. In particular, among all cases, we have $[K : F] \leq 6$.
    Since the number of extensions of a $p$-adic field with fixed degree is finite, it follows that there is a finite number of fields $K$ be considered. For each of these fields, once we compute the bound~$r$ for the conductor exponent in step (2), we obtain $r$ finite
    groups $(\cl O_K/\p^f)^{\times}$ with $1 \leq f \leq r$.
    Thus we can search for all characters of $(\cl O_K^{\times}/\p^r)$ satisfying the conditions of Theorem~\ref{T: explicit conditions p=2} or Theorem~~\ref{T: explicit conditions p=3}; this is a finite check due to Lemma~\ref{L: Is enough to check the ui}. Once we have these characters, there is a finite number of maps $\varphi_f$ through which they can factor, giving the conductor.
\end{proof}

\begin{remark}
    It is important to notice that, by searching for the inertia characters as characters $\chi^A$ of $\cl O_K^{\times}$, we ensure that $\chi$ can always be extended to a character $\theta$ of $W_K$ that produces the desired type.
\end{remark}

The above algorithm has been implemented in a {\tt Magma} package that is
available together with documentations and examples at~\cite{MAGMAFiles}. The
code implements an optimize computation of all inertia types and fields of a
given field $F$. We have also precomputed this data for all quadratics and
cubics of $\Q_2$ and $\Q_3$.

\begin{example}\label{Example:Q4}
    The following shows the computation of all inertial types (without the inertial fields) defined over~$\Q_4$ with our {\tt Magma} package. This computation was run on a consumer laptop under~$30$ seconds. The computation of the corresponding inertial fields takes less than $1$ hour. However, the complexity of the problem growths rapidly with the degree of the field $F$; for instance, for each cubic extensions of $\Q_2$, the algorithm takes about $20$ minutes to compute all inertial types and approximately $1$ day to compute all inertia fields.
    \begin{lstlisting}
> AttachSpec("spec");
> Q2 := pAdicField(2, 100);
> Q4 := UnramifiedExtension(Q2, 2);
> time PS, SCU, SCR, Ex8, Ex24 := InTypes(Q4);
Time: 27.000
> InTypesSummary(PS,SCU,SCR,Ex8,Ex24);

========================================
    Computed Inertia Types : Summary
========================================

PrincipalSeries  : 27
SCU              : 12
SCR              : 32
Ex8              : 24
Ex24             : 72
----------------------------------------
Total inertial types : 167

------------------------------------------------------------
e   v(N)  Character order  Description               Count  
------------------------------------------------------------
2   4     2                principal series          3      
2   6     2                principal series          4      
3   2     3                principal series          1      
4   6     4                principal series          4      
4   6     4                supercuspidal unramified  4      
4   8     4                supercuspidal unramified  8      
4   8     4                principal series          8      
6   4     6                principal series          3      
6   6     6                principal series          4      
8   5     4                supercuspidal ramified    4      
8   5     4                exceptional, Q8           8      
8   6     4                exceptional, Q8           16     
8   6     4                supercuspidal ramified    4      
8   8     4                supercuspidal ramified    24     
24  3     4                exceptional, SL(2,3)      3      
24  4     4                exceptional, SL(2,3)      9      
24  6     4                exceptional, SL(2,3)      12     
24  7     4                exceptional, SL(2,3)      48     
------------------------------------------------------------

    \end{lstlisting}
\end{example}

\begin{remark}\label{R: comments on algorithm}
    As the above example shows, computing the inertial types, or equivalently, the
    associated inertia characters is very fast. Computing the inertial
    fields is computationally heavier and is clearly the bottleneck of the algorithm if one seeks to compute and tabulate all inertial types over all $p$-adic fields of a given degree. This is expected, since computing the inertial fields involves finding the field fixed by the kernels of the inertia characters. For
    instance, to compute the inertial field of an exceptional type of a cubic
    extension of $\Q_2$, one needs to find fields of degree $144$. In cases where one is not
    interested in the exceptional types, the algorithm allows for the computation
    of only the non-exceptional types, which provides a significant speedup.
\end{remark}

\begin{remark}
    Since internally the package does not work with elliptic curves, it can easily be adapted to compute inertial types with more general images and conductor bounds.
\end{remark}

\begin{example} We have determined all inertial types arising from elliptic curves defined over all quadratic and cubic extensions of $\Q_p$ for $p=2$ and $p=3$. From Theorem~\ref{Inertia field determine the type}, these types are determined by their inertia fields, which we also computed. All data is available in~\cite{MAGMAFiles}.
\end{example}

\begin{remark}\label{R: Curves are bad} We want to highlight the importance of our main theorems being an equivalence. Indeed, without knowing that all types computed by our algorithm arise from elliptic curves, we would need to find an elliptic curve to realize each type as originally done in~\cite{dembélé2024galoisinertialtypeselliptic} over $\Q_p$. This could easily become a bottleneck larger than the one discussed in Remark~\ref{R: comments on algorithm}; indeed, by randomly generating curves over $\Q_4$, it took us more than one week to find curves realizing all types  computed in Example~\ref{Example:Q4}, see~\cite{GitQp2}. Obviously, when there are few inertial types left to be realized, it becomes harder to find the curves realizing them; in particular, if there is just one inertial type left to be realized after many hours, we would not know whether we have not yet found a curve realizing it or if  does not arise from elliptic curves.
\end{remark}

Finally, we end with an application of our results outside the realm of
elliptic curves. Namely, the following example is used in~\cite{PacTor} to
determine the inertial types of certain hypergeometric motives relevant for
studying the generalized Fermat equation $x^5 + y^p = z^3$.

\begin{example}\label{Example Ariel}
    Consider the following hyperelliptic curves defined over $\Q$
    $$ \cl C_1:~y^2=5x^6-12x^5-\frac{5}{4}x^3+\frac{1}{64}\qquad \text{ and }\qquad \cl C_2:~y^2=5x^6-12x^5+\frac{45}{4}x^3+\frac{81}{64}
    $$
    whose Jacobians $\cl J_i$ become of $\GL_2$-type over $K=\Q(\sqrt{5})$ as explained in~\cite{PacTor} with real multiplications by~$K$. Note that 3 is inert in $K$ and let $\rho_i: G_{\Q_9}\xrightarrow{}\GL_2(K_\p)$ be the restriction to the decomposition group at $(3)$ of the 2-dimensional $\p$-adic representation associated with~$\cl J_i / K$.
    In the proof of \cite[Theorem 7.18]{PacTor}, it is explained that
    the representations $\rho_i$ are ramified supercuspidal with trivial determinant, have inertia image $C_3\rtimes C_4$ and conductor exponent $m(\rho_i)=3$. Therefore, the inertia type of $\rho_i$ must be among the types we computed over $\Q_9$.
    Looking at the data in \cite{MAGMAFiles} for $\Q_9$, we see that there are $8$
    types with the required properties. Running through the list of corresponding inertial fields, we check over which ones the hyperelliptic curves attains good reduction. This yields the same field for both curves, defined by the polynomial $x^{12} + 3x^4 + 3$ over $\Q_9$.
    By our results, this inertial type must also arise from an elliptic curve and, indeed, it is easy to check that it is the inertial type of the base change of the curve
    $y^2=x^3-432$ to $\Q_9$.
\end{example}

\end{document}